%% file: paper.tex
\documentclass[final,onefignum,onetabnum]{siamart190516}

\usepackage{array}

\usepackage[caption=false]{subfig}

\newsiamthm{claim}{Claim}
\newsiamremark{remark}{Remark}
\newsiamremark{hypothesis}{Hypothesis}
\crefname{hypothesis}{Hypothesis}{Hypotheses}
\newsiamremark{assumption}{Assumption}
\crefname{assumption}{Assumption}{Assumption}

\usepackage[T1]{fontenc}
\usepackage{amsfonts,amsmath,amssymb}
\usepackage{mathrsfs,mathtools,stmaryrd,wasysym}
\usepackage{algorithmic}
\usepackage{algorithm}
\usepackage{enumerate}
\usepackage{xspace,./mydef} 
\usepackage[]{hyperref}
\usepackage{xcolor}
\usepackage{accents}

\usepackage{pdfsync}

\usepackage[notcite, notref]{showkeys} %

\usepackage[draft,inline,multiuser,nomargin]{fixme}
\FXRegisterAuthor{lh}{alh}{LH} 
\FXRegisterAuthor{wl}{awl}{WL}

\usepackage[textsize=small]{todonotes}

\newcommand{\reviewblue}[1]{{\textcolor{black}{#1}}}
\newcommand{\reviewred}[1]{{\textcolor{black}{#1}}}

\usepackage{xspace}
\usepackage{bold-extra}
\usepackage[most]{tcolorbox}

\colorlet{texcscolor}{blue!50!black}
\colorlet{texemcolor}{red!70!black}
\colorlet{texpreamble}{red!70!black}
\colorlet{codebackground}{black!25!white!25}

\lstdefinestyle{siamlatex}{%
  style=tcblatex,
  texcsstyle=*\color{texcscolor},
  texcsstyle=[2]\color{texemcolor},
  keywordstyle=[2]\color{texemcolor},
  moretexcs={cref,Cref,maketitle,mathcal,text,headers,email,url},
}

\tcbset{%
  colframe=black!75!white!75,
  coltitle=white,
  colback=codebackground, %
  colbacklower=white, %
  fonttitle=\bfseries,
  arc=0pt,outer arc=0pt,
  top=1pt,bottom=1pt,left=1mm,right=1mm,middle=1mm,boxsep=1mm,
  leftrule=0.3mm,rightrule=0.3mm,toprule=0.3mm,bottomrule=0.3mm,
  listing options={style=siamlatex}
}

\newtcblisting[use counter=example]{example}[2][]{%
  title={Example~\thetcbcounter: #2},#1}

\newtcbinputlisting[use counter=example]{\examplefile}[3][]{%
  title={Example~\thetcbcounter: #2},listing file={#3},#1}

\DeclareTotalTCBox{\code}{ v O{} }
{ %
  fontupper=\ttfamily\color{black},
  nobeforeafter,
  tcbox raise base,
  colback=codebackground,colframe=white,
  top=0pt,bottom=0pt,left=0mm,right=0mm,
  leftrule=0pt,rightrule=0pt,toprule=0mm,bottomrule=0mm,
  boxsep=0.5mm,
  #2}{#1}

\patchcmd\newpage{\vfil}{}{}{}
\begin{tcbverbatimwrite}{tmp_\jobname_header.tex}
\title{Adaptive finite element approximations for elliptic problems using regularized forcing data\thanks{Draft date: \today
\funding{This work was partially supported by the National Research Projects (PRIN  2017) ``Numerical Analysis for Full and Reduced Order Methods for the efficient and accurate solution of complex systems governed by Partial Differential Equations'', funded by the Italian Ministry of Education, University, and Research.}}}

\author{Luca Heltai\thanks{Mathematics Area, SISSA -- International School for Advanced Studies, via Bonomea 265, 34136, Trieste, Italy (\email{luca.heltai@sissa.it}, \email{wenyu.lei@sissa.it}).}
\and Wenyu Lei\footnotemark[2]
}

\headers{AFEM for regularized elliptic problems}{L.~Heltai and W.~Lei}
\end{tcbverbatimwrite}
\input{tmp_\jobname_header.tex}

\ifpdf
\hypersetup{ pdftitle={AFEM for regularized elliptic problems} }
\fi

\begin{document}
\maketitle

\begin{tcbverbatimwrite}{tmp_\jobname_abstract.tex}
    \begin{abstract}
        We propose an adaptive finite element algorithm to approximate solutions
        of elliptic problems whose forcing data is locally defined and is
        approximated by regularization (or mollification). We show that
        the energy error decay is quasi-optimal in two dimensional space and
        sub-optimal in three dimensional space. Numerical simulations are
        provided to confirm our findings.
    \end{abstract}


\end{tcbverbatimwrite}
\input{tmp_\jobname_abstract.tex}

\section{Introduction}\label{s:intro}
Let us consider the numerical approximation of the following elliptic problem with rough data: given a bounded domain $\Omega\subset \mathbb R^d$ with $d=2$ or $3$, we seek a distribution $u$ satisfying
\begin{equation}\label{e:distributional}
    \begin{aligned}
        -\DIV(A(x)\GRAD u) + c(x)u = F, \quad & \text{ in } \Omega,          \\
        u = 0, \quad                          & \text{ on } \partial\Omega .
    \end{aligned}
\end{equation}
Here $A(x)$ is a $d\times d$ symmetric positive definite matrix with all entries in $C^1(\overline \Omega)$. We further assume that there exist positive constants $a_0$ and $a_1$ satisfying
\begin{equation}\label{i:Ax}
    a_0|\nu|^2\le \nu^{\Tr} A(x) \nu \le a_1|\nu|^2, \quad \forall \nu\in\Rd \text{ and } x\in \overline\Omega .
\end{equation}
The lower order coefficient $c(x)$ is set to be non-negative and Lipschitz in $\overline \Omega$. We consider rough forcing data $F$ that can be written as
\[
    F(x) := \int_B \delta(x-y) f(y) \diff y, \quad\text{ with } B \subset \Omega,
\]
where $\delta$ denotes the $d$-dimensional Dirac distribution and $B\subset \mathbb R^d$ is an immersed domain. If the co-dimension of $B$ is zero, $F(x)=\chi_B(x)f(x)$ with $\chi_B$ denoting the indicator function of $B$. If the co-dimension of $B$ is one, $F$ can be written as a distribution. That is
\begin{equation}\label{e:rhs}
    \langle F, \phi\rangle := \int_B f(y) \phi(y)\diff y,\quad \forall \phi \in C^\infty_c(\overline \Omega) .
\end{equation}
In the rest of the paper, our discussion on the numerical approximation of \eqref{e:distributional} will be restricted to the co-dimension one case.

The above elliptic problem is a prototype of governing differential equations for interface problems, phase transitions and fluid-structure interactions problems using the immersed boundary method \cite{peskin2002immersed,BoffiGastaldi2021,DeSilva2015,Suarez2014}. Many works exist that concentrate on the study of (adaptive) finite element methods with point Dirac sources~\cite{Bertoluzza2017,Wihler2012,Abner2019}. The relevant literature for more complex distributions of singularities is more limited~\cite{Millar2021,Yin2021}. The motivation for such methods lies on the possibly complex geometry of the immersed domain, such as thin vascular structures in tissues \cite{heltai2019multiscale,heltai2021coupling,Zunino2019} or fibers in isotropic materials \cite{alzetta2020multiscale}, for which it is difficult to obtain a bulk mesh of $\Omega$ matching the embedded domain.

On the the other hand, when considering a non-matching bulk mesh to approximate problem \eqref{e:distributional}, it is necessary to evaluate $F$ on the quadrature points of $\Omega$ or to compute \eqref{e:rhs} when $\phi$ is a test function in a finite dimensional space. The implementation of the former strategy was introduced by Peskin in the early seventies (see \cite{peskin2002immersed} for a review) in the context of finite differences, and later adopted to  finite volume and finite element approaches \cite{mittal2005immersed}. The latter approximation strategy, usually referred to as the ``variational formulation'', was introduced in \cite{boffi2003finite} and later works, for example \cite{heltai2012variational}.

\reviewred{When computing $\int_B f\phi$ in the variational formulation, one has
    a choice to make: i) either evaluate $f$ and $\phi$ on the quadrature points
    derived from a fixed subdivision of $B$ which is independent on the subdivision
    of $\Omega$ (using a single quadrature scheme on $B$), or ii) evaluate $f$ and
    $\phi$ on the non-zero intersections of cells $K\in B$ and $T\in\Omega$ (using a
    custom quadrature formula for the generally polygonal intersection).}

\reviewred{The first approach is cheaper to compute, but it introduces some
    errors due to integration of non-smooth functions using quadrature rules. It is
    a two step process, that requires first the exact identification of the cells
    that contain quadrature points of $B$, and then the computation of the inverse
    of the mapping from the reference cell $\widehat T$ to the cell $T$ in the
    subdivision of $\Omega$ that contains the quadrature points. Such inverse
    mapping is non-linear in unstructured quad- or hex-meshes, or when using higher
    order mappings.}

\reviewred{The second approach requires a much more expensive computation, and
    its efficient implementation is the subject of active research (see,
    e.g.,~\cite{krause2016variationaltransfer, boffi2022interfacematrix}). If one
    wants to perform such integration exactly, it would require first the
    computation of the intersection between $K\in B$ and $T\in\Omega$, then the
    definition of a quadrature scheme on the (possibly polygonal or curved)
    intersection, in addition to the computation of the inverse of the mapping from
    the reference cells to the intersection part.}

\reviewred{To avoid the complexity related to the evaluation of inverse mappings and
    possibly the computation of non-matching grid intersections, here we consider an
    alternative approach by approximating $F$ with its regularization (or
    mollification) \cite{MR3429589,tornberg2002multi}. That is, we replace $\delta$
    with a family of Dirac delta approximations $\delta^r$, where $r$ denotes the
    regularization parameter so that the regularized data, denoted by $F^r$,
    satisfies certain smoothness property.}

\reviewred{In the proposed finite element algorithm, we compute a regularized
    right-hand-side $\int_B f \phi^r \,\diff x$ for a fixed parameter $r$. This
    computation requires the evaluation of the double integral
    \[
        \int_{\Omega} \int_B f(x)\delta^r(x-y) \phi(y) \diff x \diff y.
    \]
    When applying quadrature schemes on both the support of $\phi$ and on $B$, we
    only evaluate $f$ and $\phi$ on (independent sets of) quadrature points of $B$
    and of $\Omega$, respectively, weighted by the regularized Dirac distribution.
    This computation need to be performed only when the integration cells are at a
    distance smaller than $r$, and does not require any special implementation.}

The error between the exact solution $u$ and its regularized counterpart
$\uepsilon$ is analyzed in \cite{heltai2020priori} in both the $H^1$ and $L^2$
sense. The finite element approximation of \eqref{e:distributional} using
quasi-uniform subdivisions is also discussed in \cite{heltai2020priori}
\reviewred{where we also show (see \cite[Figure~7]{heltai2020priori}), that the
    computational cost and the accuracy of the regularization approach are
    comparable to the corresponding non-regularized approach, at least in the first
    case described above. The regularization in this case has the advantage of being
    trivial to implement. A fact that contributed significantly to the success of
    the immersed boundary method in the literature, which remains one of the most
    used methods in the finite difference and finite volume community for the
    computation of non-matching couplings.}

In this paper, we consider the finite element approximation of
\eqref{e:distributional} with the regularized data $F^r$ under adaptive
subdivisions. \reviewred{We show that the regularization approach is not only
    trivial to implement, but it also lends itself quite well to adaptive finite
    element methods (AFEMs) and to a-posteriori errror analysis}. AFEMs have been
widely used for decades; see \cite{nochetto2009theory} for a survey of AFEMs for
elliptic problems. In terms of the singular data $F\in H^{-1}(\Omega)$, we refer
to \cite{stevenson2007optimality, stevenson2005optimal} for piecewise constant
approximation of $F$ and \cite{cohen2012convergence} using surrogate data
indicators. \reviewblue{We also refer to
    \cite{nochetto1995pointwise,kreuzer2021oscillations} on AFEM for more complex
    singularities.}

The approximation error based on regularized data consists of two parts: the regularization error for $u$ and the finite element approximation error for $\uepsilon$. The analysis of adaptive algorithms applied to the regularized problem is complicated by the fact that optimal choices of the regularization parameter $r$ depend on the local mesh size $h$ (see~\cite{heltai2020priori}), and that the error estimates depend both on the local mesh size and on the regularization parameter $r$.

We present our algorithm in Section~\ref{s:algorithm}. We control each error in a separate routine: the routine $\INTERFACE$ controls the first error using the perturbation theory built in \cite{heltai2020priori} (see also Proposition~\ref{p:reg-error-sol}) and returns the optimal regularization parameter $r$ to use in the  routine $\SOLVE$, which controls the error of the regularized problem using classic AFEM results based on \cite{cohen2012convergence}.

Given a target tolerance, the $\INTERFACE$ routine refines \emph{a priori} the cells around the immersed domain so that the regularization error can be properly controlled. This procedure ensures that the regularization parameter $r$ is suitable for the local mesh size around the immersed domain. Given the regularization parameter $r$, the $\SOLVE$ routine will then approximate the regularized problem using AFEM based on \cite{cohen2012convergence} so that the finite element error can also be reduced below the desired tolerance. Our complete  algorithm is based on the iteration of the two routines above with a decaying target tolerance.

The performance of our adaptive algorithm is studied adapting the theories from \cite{cohen2012convergence,bonito2013adaptive} to our regularized problem. The major point to take into account is that all the estimates one obtains are generally dependent on the regularization parameter $r$, which in turn is generally chosen according to the local mesh size $h$. More precisely speaking, the following two issues must be analyzed carefully:
\begin{itemize}
    \item For any $r>0$, the regularized solutions $\uepsilon$ are in some approximation class $\mathcal A^s$ for some $s\in (0, \tfrac1d]$ (see Section~\ref{ss:solve} for the definition) and the corresponding quasi-semi-norms are uniformly bounded.
    \item Since regularized data $F^r$ is in $L^2(\Omega)$, we can guarantee
          that there exists an adaptive method to approximate $F^r$ with a
          quasi-optimal rate (cf. \cite[Assumption~$\widetilde A(s)$]{cohen2012convergence}).
          \reviewblue{That is, starting from a subdivision $\mathcal T$ and applying
              the bulk chasing strategy to obtain a refinement $\mathcal T^*$ of
              $\mathcal T$, the data indicator (defined in Section~\ref{ss:indicator}) is less than the tolerance $\tau$
              and
              \[
                  \#(\mathcal T^*) - \#(\mathcal T) \le C\tau^{-d} ;
              \]
              see \cite[Theorem~7.3]{cohen2012convergence}. However, the constant $C$ above is depending on the regularization parameter
              $r$, \ie on the local mesh size $h$ and may lead to a deterioration of the
              convergence rates.}
\end{itemize}
To resolve the first issue, we follow the arguments from
\cite{bonito2013adaptive}. Thanks to the \emph{a priori} refinements from the
$\INTERFACE$ part of the algorithm, Lemma~3.2 of \cite{bonito2013adaptive}
allows us to measure the complexity of $\SOLVE$ stage independently of $r$.
\reviewblue{ To remedy the second issue, in Lemma~\ref{l:greedy}, we revisit
    \cite[Theorem~7.3]{cohen2012convergence} and provide a finer estimate for the
    constant $C$ above which can be shown to be $C\sim r^{1-d/2}$ by exploiting the
    fact that $F^r$ is supported in the neighborhood of the immersed domain.
    It turns out that we can still obtain optimal
    convergence rates in the two dimensional case, while we get
    suboptimal rates in the three dimensional case. We show
    this in Theorem~\ref{t:convergence-rate} and Remark~\ref{r:convergence-rates}.}

The rest of this article is organized as follows. In Section~\ref{s:prelim} we provide some essential notations to define our model problem in the variational sense, and we introduce the data regularization (or data mollification) as well as a regularized version of the model problem. In Section~\ref{s:algorithm} we review the AFEM for elliptic problems with $L^2(\Omega)$ forcing data. Following this approach, we then propose our adaptive algorithm for the model problem. The analysis of the adaptive algorithm is presented in Section~\ref{s:analysis}. In Section~\ref{s:numerics} we provide some numerical experiments to illustrate the performance of our proposed algorithm. We conclude with some remarks in Section~\ref{s:discussion}.

\subsection*{Notations and Sobolev spaces}
Let $\Omega\subset \mathbb R^d$ be a bounded Lipschitz domain. We write $A\lesssim B$ if $A\le cB$ for some constant $c$ independent of $A$, $B$ as well as other discretization parameters. We say $A\sim B$ if $A\lesssim B$ and $B\lesssim A$.

Given a Hilbert space $X$, we denote with $(\cdot, \cdot)_X$ its inner product, and with $X'$ its dual space with the induced norm
\[
    \|F\|_{X'} = \sup_{\|v\|_X=1} \langle F, v\rangle_{X',X},
\]
where $\langle\cdot,\cdot\rangle_{X',X}$ denotes the duality pairing.

We indicate with $L^2(\Omega)$, $H^1(\Omega)$ and $H^2(\Omega)$ the usual Sobolev spaces and use $(\cdot,\cdot)_\Omega$ to indicate the $L^2(\Omega)$-inner product. For $s\in (0,1)$, we denote the fractional Sobolev spaces $H^s(\Omega)$ using the Sobolev--Slobodeckij norm
\[
    \|v\|_{H^s(\Omega)} :=\Big( \|v\|_{L^2(\Omega)}^2
    + \int_\Omega\int_\Omega \frac{(v(x)-v(y))^2}{|x-y|^{d+2s}}\diff x\diff y\Big)^{1/2} .
\]
For $s\in (1,2)$,
\[
    \|v\|_{H^s(\Omega)}=\big(\|v\|_{L^2(\Omega)}^2+\|\GRAD v\|_{H^{s-1}(\Omega)}^2\big)^{1/2} .
\]
For $s\in (\tfrac12, 1]$, we set $H^s_0(\Omega)$ to be the collection of functions in $H^1(\Omega)$ vanishing on $\partial\Omega$. It is well known that $H^s_0(\Omega)$ is the closure of $C_c^\infty(\overline \Omega)$ (the space of infinitely differentiable functions with compact support in $\overline \Omega$) with respect to the norm of $H^s(\Omega)$ (cf. \cite{grisvard2011elliptic}). Also, $H^s_0(\Omega)$ is an interpolation space between $L^2(\Omega)$ and $\Hunz$ using the real method. Finally for $s\in (\tfrac12 ,1]$, we set $H^{-s}(\Omega) = H^s_0(\Omega)'$.

\section{Model problem and its regularization}\label{s:prelim}
In this section, we will introduce the variational formulation of our model problem as well as a formulation when the forcing data $F$ is approximated by regularization.

\subsection{The forcing data}
Let $\omega\subset\Omega$ be a bounded domain and let $\gamma := \partial
    \omega$ be its boundary, which \reviewblue{we take to be Lipschitz}. In what
follows, we only consider the case when $\gamma$ is away from $\partial\Omega$,
\ie there exists a positive constant $\constBoundary$ such that
\begin{equation}\label{i:constBoundary}
    \dist(\gamma,\partial\Omega) > \constBoundary.
\end{equation}

\reviewred{We assume that the data function $f\in L^\infty(\gamma)$}. \reviewblue{For a technicality (cf. Lemma~\ref{l:quasi-monotonicity})},
\reviewred{we further assume that there exists a finite collection of non-overlapping
non-empty open sets $\{\gamma_j \subset \gamma\}_{j=1}^{M_\gamma}$ such that $\sum_{j=1}^{M_\gamma} |\gamma_j| = |\gamma|$ and $f$
does not change sign on each $\gamma_j$.}
\reviewblue{ We define $I$ to be the set where $f$ changes sign, i.e.,
    \begin{equation}\label{e:sign-change}
        I :=
        \bigcup_{i=1}^{M_\gamma} \partial \gamma_i .
    \end{equation}
    The above limitation on the
    sign change is used only in Lemma~\ref{l:quasi-monotonicity}, and allows
    us to simplify its proof, without sacrificing too much on the generality
    of the admissible data. In particular, a sufficient condition for the
    above statement to be true is that the co-dimension two measure of $I$
    is bounded, i.e., $I$ consists of a finite number of points for
    one-dimensional curves embedded in two dimensions, or collection of
    curves with finite length for two-dimensional surfaces embedded in
    three-dimension.}

We then consider a forcing data that can be formally written as
\begin{equation}\label{e:dirac}
    F = \mathcal M f := \int_{\gamma} \delta(x-y) f(y) \diff \sigma_y.
\end{equation}
The variational definition of $F$ (see, \ie \cite{heltai2020priori}) implies that that $F\in H^{-s}(\Omega)\subset H^{-1}(\Omega)$ with any fixed $s\in (\tfrac12, 1]$. In fact, for any $v\in \Hunz$, there holds
\begin{equation}\label{i:bounded-functional}
    \begin{aligned}
        \langle F, v\rangle_{H^{-1}(\Omega),\Hunz} & = \int_\gamma fv\diff\sigma                      \\
                                                   & \le \|f\|_{L^2(\gamma)} \|v\|_{L^2(\gamma)}      \\
                                                   & \lesssim \|f\|_{L^2(\gamma)} \|v\|_{H^s(\omega)}
        \lesssim \|f\|_{L^2(\gamma)} \|v\|_{H^s(\Omega)} ,
    \end{aligned}
\end{equation}
where for the first inequality above we applied \reviewblue{Schwarz} inequality and for the second inequality we used the trace inequality.

\subsection{Weak formulation}
The variational formulation of \eqref{e:distributional} reads: given a function $f\in L^2(\gamma)$, we seek $u\in \Hunz$ such that
\begin{equation}\label{e:variational}
    A(u,v) = \langle F,v \rangle_{H^{-1}(\Omega), \Hunz},\quad\forall v\in \Hunz ,
\end{equation}
where
\[
    A(v,w) = \int_\Omega \GRAD v^{\Tr} A(x)\GRAD w + c(x)vw\diff x,
    \quad \forall v,w\in \Hunz.
\]
Assumption \eqref{i:Ax} and the non-negativity of $c(x)$ guarantee that the bilinear form $A(\cdot,\cdot)$ is bounded and coercive, \ie there exist positive constants $m, M$ so that for $v,w\in\Hunz$,
\begin{equation}\label{i:bound-coercive}
    A(v,w)\le M\|v\|_{H^1(\Omega)}\|w\|_{H^1(\Omega)}\text{ and }
    A(v,v)\ge m\|v\|^2_{H^1(\Omega)},
\end{equation}
and \eqref{e:variational} admits a unique solution by \reviewred{the} Lax-Milgram Lemma. Bound \eqref{i:bound-coercive} also
implies that the energy norm $\vertiii{v} := \sqrt{A(v,v)} \sim \|v\|_{H^1(\Omega)}$. In what follows, we use the energy norm $\vertiii{.}$ instead of $\|.\|_{H^1(\Omega)}$ in our adaptive algorithm as well as in the performance analysis.
\subsection{Regularization}
The regularization of $F$ is based on the approximation of the Dirac delta distribution. To this end, we first define a class of functions $\psi$ satisfying the following assumptions:
\begin{assumption}\label{a:app-dirac}
    Given $k \in \mathbb{N}$, let $\psi(x)$ in $L^\infty(\Rd)$
    such that
    \begin{enumerate}[1.]
        \item \textbf{Nonnegativity:} $\psi(x)\ge 0$;
        \item \textbf{Compact support:}
              $\psi(x)$ is compactly supported, with support $\mathrm{supp}(\psi)$
              contained in $B_{r_0}(0)$ (the ball centered in zero with radius $r_0$)
              for some $r_0>0$;
        \item \textbf{Moments condition}: Given $k\in\mathbb N$, we say $\psi$ satisfies the $k$-th
              order moment condition if
              \begin{equation}
                  \int_\Rd y^\alpha_i\psi(x-y)\diff y = x^\alpha_i \qquad
                  i=1\ldots d, \quad 0\leq\alpha\leq k, \quad\forall x \in \Rd;
              \end{equation}
        \item \reviewred{\textbf{Monotonicity:} $\psi(x/r_2) \leq \psi(x/r_1)$ if $r_2<r_1$.}
    \end{enumerate}
\end{assumption}
We refer to \cite{heltai2020priori} for some examples of $\psi$ and
\cite[Section~3]{MR3429589} for a general discussion. \reviewred{Here we only
    consider even, nonnegative functions $\psi_{1d}$ that are supported in $[-1,1]$,
    are nonincreasing in $[0,1]$, and satisfy $\int_{\mathbb R}\psi_{1d}=1$. Then we
    generate $\psi$ in $\Rd$ by the radially symmetric extension $\psi_{1d}(|x|)$ or
    the tensor product extension $\prod_{i=1}^d\psi_{1d}(x_i)$. A function $\psi$
    defined in this way satisfies Assumption~\ref{a:app-dirac} with $k=1$.} Using
the above $\psi$, for $r>0$, we define the Dirac approximation $\delta^r$ by
\begin{equation}
    \label{e:dirac-approximation}
    \depsilon(x) := \frac{1}{r^d} \psi\left(\frac{x}{r}\right).
\end{equation}
Thus,
\[
    \lim_{r \to 0} \depsilon(x) = \lim_{r \to 0}
    \frac{1}{r^d} \psi\left(\frac{x}{r}\right) = \delta(x),
\]
where the limit should be understood in the space of \reviewblue{Schwarz} distributions.

\begin{remark}[nonnegativity of $\psi$]
    We will use the non-negativity of $\psi$ to analyze the performance of our adaptive algorithm. However, this is not required in the error analysis for finite element discretization of \eqref{e:distributional} using quasi-uniform subdivisions of $\Omega$; see \cite{heltai2020priori} for more details.
\end{remark}

\begin{definition}[Regularization]
    \label{d:regularization}
    For a function $v\in L^1(\Omega)$\, we define its regularization $\vepsilon(x)$ in the domain $\Omega$ through the mollifier $\psi$ by
    \begin{equation}\label{e:regularization-A}
        \vepsilon(x) := \int_{\Omega} \depsilon(x-y) v(y) \diff y,\qquad\forall x\in \Omega,
    \end{equation}
    where $\depsilon$ is given by \eqref{e:dirac-approximation} and where $\psi$ satisfies Assumption~\ref{a:app-dirac} for some $k\geq 0$.

    For functionals $F$ in negative Sobolev spaces, say $F\in H^{-s}(\Omega)$, with $s\in (\tfrac12,1]$, we define their regularization $F^r$ by the action of $F$ on $v^r$ with $v\in H^s_0(\Omega)$, \ie
    \begin{equation}\label{e:regularization-B}
        \langle F^r, v \rangle_{H^{-s}(\Omega), H^s_0(\Omega)} :=   \langle
        F, v^r \rangle_{H^{-s}(\Omega), H^s_0(\Omega)} .
    \end{equation}
\end{definition}
We note that the definition of $F^r$ is well defined with $F$ given by \eqref{e:dirac}. In fact, by \cite[Corollary~1]{heltai2020priori}, there holds
\[
    \|v-v^r\|_{H^s(\omega)} \lesssim \|v\|_{H^s(\Omega)}.
\]
Therefore, according to the argument in \eqref{i:bounded-functional}, we have
\[
    \langle F^r, v \rangle_{H^{-s}(\Omega), H^s_0(\Omega)}
    \lesssim \|f\|_{L^2(\gamma)} \|\vepsilon\|_{H^s(\omega)}
    \lesssim \|f\|_{L^2(\gamma)}\|v\|_{H^s(\Omega)} .
\]
\begin{remark}\label{r:regF}
    For $F$ defined by \eqref{e:dirac}, applying Fubini's Theorem to the right
    hand side of \eqref{e:regularization-B} yields
    \[
        F^r(x) = \int_\gamma f(y) \depsilon(y-x) \diff y \in L^2(\Omega) .
    \]
    If $\psi$ is chosen to be symmetric, the definition of $F^r$ can be interpreted by replacing $\delta$ in \eqref{e:dirac} with the Dirac approximation $\depsilon$.
\end{remark}

\begin{remark}[Error estimate of the regularization]\label{r:data-error}
    \reviewred{Lemma~10 of \cite{heltai2020priori} implies that under the
        Assumption~\ref{a:app-dirac}, together with
        Equation~\eqref{i:constBoundary}, the following regularization error
        estimate holds when $r<1$,}
    \begin{equation}\label{i:data-error}
        \|F-F^r\|_{H^{-1}(\Omega)} \le \constReg r^{1/2}\|f\|_{L^2(\gamma)},
    \end{equation}
    where the constant $\constReg$ depends on $\psi$ in
    Assumption~\ref{a:app-dirac} and on $\omega$.
\end{remark}

\subsection{Regularized problem}
A regularized version of problem \eqref{e:variational} reads: find $\uepsilon\in \Hunz$ satisfying
\begin{equation}\label{e:regularized}
    A(\uepsilon, v) = \langle F^r, v\rangle_{H^{-1}(\Omega), \Hunz}, \qquad \forall v\in \Hunz .
\end{equation}
Notice that $\uepsilon$ exists and is unique. Moreover, \eqref{i:bound-coercive}
and Remark~\ref{r:data-error} imply that $\uepsilon$ converges to $u$ in the energy norm with the rate $O(r^{1/2})$. That is
\begin{proposition}[see also Theorem~14 of \cite{heltai2020priori}]\label{p:reg-error-sol}
    When Assumption~\ref{a:app-dirac} holds, let $u$ and $\uepsilon$ be the solution to \eqref{e:variational} and \eqref{e:regularized}, respectively. Then there holds
    \begin{equation}\label{i:error-reg-sol}
        \vertiii{u-\uepsilon} \le  m^{-1/2}\constReg r^{1/2} \|f\|_{L^2(\gamma)} .
    \end{equation}
\end{proposition}

\section{Numerical algorithm}\label{s:algorithm} We approximate the solution to
the weak formulation \eqref{e:variational} by solving the regularized problem
\eqref{e:regularized} using AFEMs along with a choice of the regularization
parameter $r$. As the number of degrees of freedom increase, $r$ will tend to
zero \emph{with a rate linked to the target tolerance}. Recalling from
Remark~\ref{r:regF}, the regularized data $F^r$ is an $L^2(\Omega)$ function so
that we can use classical residual error estimators for adaptivity. In this
section, we first review AFEMs for elliptic problems with $L^2(\Omega)$ forcing
data based on
\cite{cohen2012convergence,stevenson2007optimality,nochetto2009theory}. Then we
introduce our adaptive algorithm for \eqref{e:variational}.

\subsection{Finite element approximation}
We additionally assume that $\Omega$ is a polytope. Given a data function $g\in L^2(\Omega)$, we consider a finite element approximation of $w_g \in \Hunz$ which uniquely solves
\begin{equation}\label{e:linear-problem}
    A(w_g, v) = (g, v),\quad \forall v\in \Hunz .
\end{equation}

Set $\mathcal T$ to be a subdivision of $\Omega$ made by simplices. We assume that $\mathcal T$ is conforming (no hanging nodes) and shape-regular in a sense of \cite{ern2013theory,ciarlet2002finite}, \ie there exists a positive constant $\constShapeRegular$ so that for each cell $T\in \mathcal T$,
\[
    \diam(T) \le \constShapeRegular \rho_T
\]
with $\diam(T)$ and $\rho_T$ denoting the size of $T$ and the diameter of the largest ball inscribed in $T$, respectively. We also set $h_T = |T|^{1/d}$, with $|T|$ denoting the volume of $T$. So $h_T\sim \diam(T)$, with the hiding constants depending on $\constShapeRegular$. Denote $\mathbb V(\mathcal T)\subset \Hunz$ the space of continuous piecewise linear functions subordinate to $\mathcal T$. So the finite element discretization for \eqref{e:linear-problem} reads:
\begin{algorithm}
    \begin{algorithmic}
        \STATE Solve $A(W_g, V) = (g,V),\quad\forall V\in \mathbb V(\mathcal T)$;
        \RETURN $W_g$;
    \end{algorithmic}
    \caption{$W_g = \GAL(\mathcal T, g)$}
    \label{alg:gal}
\end{algorithm}

\subsection{A posteriori error estimates with \texorpdfstring{$L^2(\Omega)$}{} data}\label{ss:indicator}
AFEMs rely on the so-called computable error estimators to evaluate the quality of the finite element approximation on each cell $T$ in the underlying subdivision $\mathcal T$. Here we consider the following local jump residual and data indicators: given a conforming subdivision $\mathcal T$, a finite element function $V\in \mathbb V(\mathcal T)$ and a data function $g\in L^2(\Omega)$, we denote $\mathcal F_T$ the collection of all faces of $T\in \mathcal T$ and define
\begin{equation}\label{e:jump-residual}
    j(V,T,\mathcal T) := \bigg( \sum_{F\in\mathcal F_T} h_F \|[A\cdot \GRAD V]\|_{L^2(F)}^2 \bigg)^{1/2}
    \text{ and }
    d(g,T,\mathcal T) := h_T \|g\|_{L^2(T)} ,
\end{equation}
where $h_F$ is the size of $F$ and $[.]$ denotes the normal jump across the face $F$. Their global counterparts are given by
\[
    \mathcal J(V,\mathcal T) := \bigg( \sum_{T\in\mathcal T} j(V,T,\mathcal T)^2\bigg)^{1/2}
    \text{ and }
    \mathcal D(g, \mathcal T) :=  \bigg( \sum_{T\in\mathcal T} d(g,T,\mathcal T)^2\bigg)^{1/2} .
\]
Letting $W_g = \GAL(\mathcal T, g)$, we define the local error indicator:
\[
    e(W_g,T,\mathcal T) = \big(j(W_g,T,\mathcal T)^2 + d(g,T,\mathcal T)^2\big)^{1/2}
\]
as well as the global indicator
\[
    \mathcal E(W_g,\mathcal T) = \bigg(\sum_{T\in\mathcal T} e(W_g, T,\mathcal T)^2\bigg)^{1/2} .
\]
The computation of such indicators is usually performed in the stage $\ESTIMATE$ of AFEM algorithms, as summarized in Algorithm~\ref{alg:estimate}.

\begin{algorithm}
    \begin{algorithmic}
        \STATE Given the approximate solution $W_g$ on $\mathcal T$;
        \FOR{$T \in \mathcal T$}
        \STATE Compute $j(T) = j(W, T, \mathcal T)$;
        \STATE Compute $d(T) = d(g, T, \mathcal T)$;
        \STATE Compute $e(T) = e(W_g, T, \mathcal T)$;
        \ENDFOR
        \RETURN $\{j(T), d(T), e(T)\}_{T\in\mathcal T}$;
    \end{algorithmic}
    \caption{$\{j(T), d(T), e(T)\}_{T\in\mathcal T} = \ESTIMATE(\mathcal T, W_g)$}
    \label{alg:estimate}
\end{algorithm}

\subsection{Marking of cells based on error indicators}
\label{ss:mark}

The estimated error per cell obtained in the $\ESTIMATE$ algorithm are used to perform refinement based on  the bulk chasing strategy \cite{dorfler1996convergent} (or the D\"orfler marking strategy), summarized in Algorithm~\ref{alg:mark}. Here we set $\textrm{ind}(T)$ to be a local indicator and the corresponding global indicator is denoted by $\textrm{IND}$.

\begin{algorithm}
    \begin{algorithmic}
        \STATE Given a cell indicator $\{\textrm{ind}(T)\}_{T\in \mathcal T}$ and a bulk parameter $\theta \in (0,1)$;
        \STATE Find \reviewred{a} smallest subset $\mathcal M$ of $\mathcal T$ satisfying
        \begin{equation}\label{i:dorfler}
            \bigg( \sum_{T\in \mathcal M} \textrm{ind}(T)^2 \bigg)^{1/2} \ge \theta \,\textrm{IND} .
        \end{equation}
        \RETURN $\mathcal M$;
    \end{algorithmic}
    \caption{$\mathcal M = \MARK(\{\textrm{ind}(T)\}_{T\in \mathcal T}, \mathcal T, \theta)$}
    \label{alg:mark}
\end{algorithm}

\subsection{Refinements of subdivisions}\label{sss:refine}
\reviewblue{Conforming refinement strategies, such as newest vertex bisection
\cite{binev2004adaptive,mitchell1989comparison,stevenson2008completion}, can be
used to construct a sequence of conforming simplicial subdivisions $\{\mathcal
    T_k\}_{k=0}^\infty$ by adaptively bisecting a set of cells $\mathcal
    R_k\subset\mathcal T_k$. However, our results hold also for more general
nonconforming mesh refinement strategies satisfying Condition~3 (successive
subdivisions), 4 (complexity of refinement) and 7 (admissible subdivision) in
\cite{bonito2010quasi}. For instance, in our numerical illustration in
Section~\ref{s:numerics}, we use refinements on quad- and hex-meshes where
conformity is enforced via hanging node constraints. Irrespective of the
strategy used to refine the grid (either conforming or nonconforming with
hanging node constraints), we obtain a sequence of uniformly
shape-regular subdivisions $\{\mathcal T_k\}_{k\ge 0}$ satisfying
\begin{equation}\label{i:complexity-refine}
    \#(\mathcal T_k) - \#(\mathcal T_0) \le \constComplexity \sum_{j=0}^{k-1} \#(\mathcal R_j),
\end{equation}
for some universal constant $\constComplexity\ge 1$.} We write the above
refinement process from $\mathcal T_k$ to $\mathcal T_{k+1}$ as $\mathcal
    T_{k+1} = \REFINE(\mathcal T_k, \mathcal R_k)$, summarized in
Algorithm~\ref{alg:refine}.

\begin{algorithm}
    \begin{algorithmic}
        \STATE i) (for triangular or tetrahedral meshes) bisect the marked cells $\mathcal R_k\subset \mathcal T_{k}$ once;
        \STATE Add all extra bisections to produce a conforming subdivision $\mathcal T_{k+1}$;
        \reviewblue{\STATE ii) (for quadrilateral or hexahedral meshes) split the marked
            cells into four children in two dimensions or eight children in three
            dimensions;
            \STATE  Refine all extra cells to produce a nonconforming subdivision
            $\mathcal T_{k+1}$ with at most one hanging node per face, and enforce
            conformity via \emph{hanging node constraints};}
        \RETURN $\mathcal T_{k+1}$;
    \end{algorithmic}
    \caption{$\mathcal T_{k+1} = \REFINE(\mathcal T_k, \mathcal R_k)$}
    \label{alg:refine}
\end{algorithm}

\subsubsection{Overlay of two subdivisions} Providing that both $\mathcal T_1$ and $\mathcal T_2$ are refinements of $\mathcal T_0$, we say that $\mathcal T$ is the \emph{overlay} of $\mathcal T_1$ and $\mathcal T_2$ when $\mathcal T$ consists of the union of all cells of $\mathcal T_1$ that do not contain smaller cells of $\mathcal T_2$ and vice versa. Clearly, there holds
\begin{equation}\label{i:overlay}
    \#(\mathcal T) \le \#(\mathcal T_1) + \#(\mathcal T_2) -\#(\mathcal T_0) .
\end{equation}

\subsection{AFEM with control on \texorpdfstring{$L^2$}{} data}
\label{ss:data}
It is well known (see \eg \cite{bank1985some, dorfler2002small,cohen2012convergence}) that one can obtain a global upper and lower bound   of the approximation error by the error indicator, \ie there exist positive constants $\constRel$ and $\constEff$ so That
\begin{equation}\label{i:upper-lower-bound}
    \vertiii{w_g-W_g}\le \constRel \mathcal E(W_g,\mathcal T) \text{ and }
    \mathcal E(W_g,\mathcal T) \le \constEff E(w_g,\mathcal T)
\end{equation}
with
\begin{equation}\label{e:total-error}
    E(w_g,\mathcal T) : = \big(\vertiii{w_g-W_g}^2 + \mathcal D(g, \mathcal T)^2\big)^{1/2}.
\end{equation}

\begin{remark}[Local lower bound with oscillation]
    The data indicator $\mathcal D(g, \mathcal T)$ in the lower bound can be replaced by the data oscillation \reviewblue{provided that the refinement strategy satisfies the interior node property \cite{bonito2013adaptive,cohen2012convergence, nochetto2009theory, morin2000data}:}
    \[
        \osc(g,\mathcal T) =\bigg(\sum_{T\in \mathcal T} h_T^2\|g - a_T(g)\|_{L^2(T)}^2\bigg)^{1/2} ,
    \]
    where $a_T(.)$ denotes the average on $T$. Note that $\osc(g,\mathcal T)\le
        \mathcal D(g, \mathcal T)$, and the decay of the data oscillation could be
    faster if $g$ is more regular. \reviewred{However, in our case, we set  $g=F^r$ to be
        as in Definition~\ref{d:regularization}, and the smoothness of $g$ depends
        on the choice of $\psi$ in Assumption~\ref{a:app-dirac} as well as the regularization parameter $r$. In order to simply our analysis, we will treat $F^r$ as an $L^2(\Omega)$ function and the decay rate of oscillation is then the same as the data indicator $\mathcal D(g,\mathcal T)$.}
\end{remark}

\reviewred{The \DATA routine guarantees that the global data indicator $\mathcal D$ is below a user defined tolerance. This allows us to control the total error indicator $\mathcal E$.}

\begin{algorithm}
    \begin{algorithmic}
        \STATE $\mathcal T^* = \mathcal T$;
        \WHILE{$\mathcal D(g, \mathcal T^*) >\tau$}
        \STATE $\mathcal M = \MARK(\{d(g, T, \mathcal T^*)\}_{T\in \mathcal T}, \mathcal T^*, \widetilde\theta)$;
        \STATE $\mathcal T^* = \REFINE(\mathcal T^*,\mathcal M)$;
        \ENDWHILE
        \RETURN \reviewred{$\mathcal T^*$;}
    \end{algorithmic}
    \caption{$\mathcal T^* = \DATA(\mathcal T, g, \tau, \widetilde\theta)$}
    \label{alg:data}
\end{algorithm}

\subsection{AFEM algorithm for \texorpdfstring{$L^2$}{} data}
\label{s:solve}

To summarize the above steps in a complete AFEM algorithm, we follow \cite{cohen2012convergence} to solve problem \eqref{e:linear-problem} by iteratively generating refined subdivisions and the corresponding finite element approximations. For convenience, we denote $W_k\in \mathbb V_k:= \mathbb V(\mathcal T_k)$ the finite element approximation of $w_g$ on $\mathcal T_k$. Similarly, we denote the local indicators $j_k(T):=j(W_k,T,\mathcal T_k)$, $d_k(T):=d(g,T,\mathcal T_k)$, $e_k(T):=e(W_k, T,\mathcal T_k)$ and global indicators $\mathcal J_k:=\mathcal T(W_k,\mathcal T_k)$, $\mathcal D_k:=\mathcal D(g,\mathcal T_k)$, $\mathcal E_k:=\mathcal E(W_k,\mathcal T_k)$.

Starting from a conforming subdivision $\mathcal T_0$ and given a tolerance $\tau>0$, we choose $\theta,\widetilde{\theta},\lambda\in (0,1)$ and construct the approximation $U_k$ by the routine $\SOLVE$, defined in Algorithm~\ref{alg:solve}.

\begin{algorithm}
    \label{alg:solve}
    \begin{algorithmic}
        \STATE $W_0 = \GAL(\mathcal T_0,g)$;
        \STATE \reviewred{$\{j_0(T), d_0(T), e_0(T)\}_{T\in\mathcal T_0} = \ESTIMATE(\mathcal T_0,W_0,g)$;}
        \STATE $k=0$;
        \WHILE{$\mathcal E_k >\tau$}
        \IF {$\mathcal D_k > \sigma_k :=\lambda\theta \mathcal E_k$}
        \STATE $\mathcal T_{k+1} = \DATA(\mathcal T_k, g,\tfrac{\sigma_k}2, \widetilde\theta)$;
        \ELSE
        \STATE \reviewred{$\mathcal M_k = \MARK(\{e_k(T)\}_{T\in \mathcal T_k}, \mathcal T_k, \theta)$;}
        \STATE $\mathcal T_{k+1} = \REFINE(\mathcal T_k,\mathcal M_k)$;
        \ENDIF
        \STATE $k = k+1$;
        \STATE $W_k = \GAL(\mathcal T_k,g)$;
        \STATE \reviewred{$\{j_k(T), d_k(g,T), e_k(T)\}_{T\in\mathcal T_k} = \ESTIMATE(\mathcal T_k,W_k,g)$;}
        \ENDWHILE
        \RETURN $\{W_{k}, \mathcal T_k\}$;
    \end{algorithmic}
    \caption{$\{W^*,\mathcal T^*\}=\SOLVE(\mathcal T_0, g, \tau, \theta, \widetilde\theta, \lambda)$}
\end{algorithm}

According to \cite{cohen2012convergence}, the routine $\SOLVE$ guarantees the decay of the error indicator $\mathcal E_k$ with some decay factor $\alpha\in (0,1)$ (see also Theorem~\ref{t:solve-contraction}) and hence, when this routine terminates, we obtain that
\begin{equation}\label{i:afem-error-control}
    \vertiii{w_g-W_g} \le \constReg \tau .
\end{equation}
Here we applied the upper bound in \eqref{i:upper-lower-bound}.

\reviewred{
    \begin{remark}[An alternative AFEM algorithm]
        In Section~\ref{s:analysis}, we will adapt the approximation theory developed in \cite{cohen2012convergence} to investigate the performance of \SOLVE for the regularized problem \eqref{e:regularized}. On the other hand, we could instead apply the classical AFEM cycle:
        \[
            \GAL \to \ESTIMATE \to \MARK \to \REFINE .
        \]
        We note that the same performance in terms of tolerances can be obtained by following the arguments from \cite{nochetto2009theory,bonito2013adaptive, cascon2008quasi} together with approximation properties developed in Section~\ref{s:analysis} (cf. Corollary~\ref{c:interface} and \ref{c:data}). However, the classical AFEM algorithm would suffer from a higher computational cost related to the higher number of \GAL steps that are computed in classical AFEM, and therefore we proceed as in algorithm~\ref{alg:solve}, following the steps of~\cite{cohen2012convergence}.
    \end{remark}
}

\subsection{AFEM algorithm for regularized \texorpdfstring{$H^{-1}$}{} data} Let
us first provide an assumption on the initial subdivision $\mathcal T_0$ related
to the interface $\gamma$. \reviewblue{We denote with
    \[
        \mathcal G := \mathcal G(\gamma,\mathcal T) := \{T\in\mathcal T : T\cap \gamma \neq \emptyset\}
        \text{ and } \diam(\mathcal G) = \max_{T\in \mathcal G} h_T,
    \]
    and we assume that the initial subdivision is sufficiently refined to capture
    the characteristics of $\gamma$, that is, $\mathcal G(\gamma,\mathcal T_0)$ is quasi-uniform and for any uniform refinement $\mathcal T_i$ of level
    $i$ of $\mathcal T_0$, we have that
    \begin{equation}\label{a:initial-grid}
        \sum_{T\in \mathcal G(\gamma,\mathcal T_i)} |T| \sim q^{-i/d} \sum_{T\in \mathcal G(\gamma,\mathcal T_0)} |T| \sim q^{-i/d} |\gamma|,
    \end{equation}
    where $q>1$ is the volume ratio between a cell and its children. In two dimensional space, for instance, $q=2$ for the newest vertex bisection and $q=4$ for the quad-refinement.The above assumption shows that there exists a positive const $c$ depending on $\constShapeRegular$ such that the narrow band of $\gamma$ with width $c q^{-i/d}$ covers $\mathcal G(\gamma,\mathcal T_i)$.}

\reviewblue{
    \begin{remark}
        Condition~\ref{a:initial-grid} is a way to ask that the initial subdivision
        $\mathcal T_0$ properly resolves $\gamma$. This is possible for Lipschitz
        curves and surfaces, and requires that the initial subdivision $\mathcal
            T$ is sufficiently refined around $\gamma$, with a local grid size that will
        depend on the Lipschitz constant of $\gamma$.
    \end{remark}
}

Given a target tolerance $\tau>0$, we shall determine the regularization parameter $r$ and approximate problem \eqref{e:linear-problem} with $g=F^r$ via $\SOLVE$ so that the output approximation $U$ satisfies
\[
    \vertiii{u-U} \le \vertiii{u-\uepsilon} + \vertiii{\uepsilon - U} \le \constReg \tau .
\]
To control $\vertiii{u-\uepsilon}$, in view of Proposition~\ref{p:reg-error-sol}, we can set
\[
    m^{-1/2}\reviewblue{\constReg} r^{1/2}\|f\|_{L^2(\gamma)} \le \frac{\constRel \tau}{2} .
\]
Hence we choose the regularization parameter
\begin{equation}\label{i:choice-regularization}
    r =: r(\tau) := \reviewblue{m}\Bigg( \frac{\constRel \tau}{2\constReg \|f\|_{L^2(\gamma)}} \Bigg)^{2} .
\end{equation}

\begin{remark}[values of the constants in \eqref{i:choice-regularization}]
    Since it is non-trivial to compute the constants that appear in \eqref{i:choice-regularization}, in the simulations presented in Section~\ref{s:numerics} we select $r = \tau^2$.
\end{remark}

\reviewred{From the computational point of view, if $r\ll h_T$ for $T\in \mathcal
        G(\gamma,\mathcal T)$ and if $V\in\mathbb V(\mathcal T)$ is nonzero in $T$,
    it is possible that $\delta^r(q_1-q_2)=0$ when $q_1$ is a quadrature
    point on $\gamma$ and $q_2$ is a quadrature point in $T$. In such case, we would
    approximate $\int_T F^r V$ by zero using the quadrature scheme, resulting in a
    ``transparent'' $\gamma$, implying a total loss of accuracy.} In order to avoid
such situation, we also refine the subdivision before controlling the error
$\vertiii{\uepsilon - U}$ from \SOLVE. Our goal is to find a refinement
$\mathcal T^*$ of $\mathcal T$ so that
\[
    2\diam(\mathcal G(\gamma,\mathcal T^*)) \le r  .
\]
To this end, we introduce the routine $\INTERFACE$ in Algorithm~\ref{alg:interface}.

\begin{algorithm}[H]
    \begin{algorithmic}
        \STATE $\mathcal T^* =\mathcal T$;
        \WHILE{$\diam(\mathcal G) > \tfrac r2$}
        \STATE \reviewred{Find the set $\mathcal M := \{ T \in \mathcal G(\gamma, \mathcal T) \text{ s.t. } h_T  > \tfrac r2\}$;}
        \STATE $\mathcal T^* = \REFINE(\mathcal T^*,\mathcal M)$;
        \ENDWHILE
        \RETURN $\mathcal T^*$;
    \end{algorithmic}
    \caption{$\mathcal T^* = \INTERFACE(\mathcal T, r)$}
    \label{alg:interface}
\end{algorithm}

Given an initial conforming subdivision $\mathcal T_0$ satisfying assumption \eqref{a:initial-grid}, an initial tolerance $\tau_0$, and $\beta,\theta,\widetilde\theta,\lambda,\widetilde\mu \in (0,1)$, the solver routine $\REGSOLVE$ for \eqref{e:variational} reads as in Algorithm~\ref{alg:regsolve}.

\begin{algorithm}
    \begin{algorithmic}
        \FOR{$j = 0 : j_{\max}$}
        \STATE $r_j = r(\tau_j)$;
        \STATE $\widetilde {\mathcal T}_j = \INTERFACE(\mathcal T_j, r_j)$;
        \STATE $\{U_{j+1},\mathcal T_{j+1}\}=\SOLVE(\widetilde {\mathcal T}_j, F^{r_j}, \widetilde\mu \tau_j, \theta,\widetilde\theta,\lambda)$;
        \STATE $\tau_{j+1} = \beta\tau_j$;
        \STATE $j = j+1$;
        \ENDFOR
        \RETURN $\{U_{j_{\max}}, \mathcal T_{j_{\max}}\}$;
    \end{algorithmic}
    \caption{$\{U,\mathcal T\}=\REGSOLVE(g, \mathcal T_0, j_{\max}, \tau_0, \beta,\theta,\widetilde\theta,\lambda,\widetilde\mu)$}
    \label{alg:regsolve}
\end{algorithm}

\noindent Here $\widetilde \mu$ is a constant whose choice will be explained later in Lemma~\ref{l:marked-cells-refine}. Note that the subroutine $\SOLVE$ in $\REGSOLVE$ guarantees that the energy error between $U$ and $\uepsilon$ is bounded by $\widetilde\mu \tau_j$. Whence,
\begin{proposition}\label{p:error-control}
    Let $u$ and $U_j$ be defined as in \eqref{e:variational} and $\emph{\REGSOLVE}$, respectively. Then for each nonnegative integer $j$,
    \[
        \begin{aligned}
            \vertiii{u-U_{j+1}} & \le \vertiii{u-u^{r_{j}}} + \vertiii{u^{r_{j}}-U_{j+1}}                                   \\
                                & \lesssim \vertiii{u-u^{r_{j}}} + \mathcal E(u^{r_j},\mathcal T_{j+1}) \lesssim \tau_{j} .
        \end{aligned}
    \]
\end{proposition}

\begin{remark}[another algorithm]\label{r:another-regsolve}
    Since $\INTERFACE$ is an \emph{a-priori} process, we can also solve \eqref{e:variational} with only one iteration in $\REGSOLVE$. That is
    \[
        \{U,\mathcal T\} = \REGSOLVE(g,\mathcal T_0, 1, \tau, \cdot, \theta,\widetilde\theta,\lambda,
        \widetilde\mu)
    \]
    \reviewblue{with $\tau = \tau_0 \beta^{j_{\max}}$.}
\end{remark}

\section{Measuring the performance}\label{s:analysis}
In this section we measure the performance of $\REGSOLVE$, \ie we analyze the subroutines $\INTERFACE$ and $\SOLVE$ respectively. We use notation $\mathrm{::}$ to connect a routine and its subroutine. For instance, the routine $\SOLVE$ in $\REGSOLVE$ is denoted by $\REGSOLVE\mathrm{::}\SOLVE$.

\subsection{Performance of \texorpdfstring{$\INTERFACE$}{INTERFACE}} The
following proposition provides the performance of $\INTERFACE$.
\begin{proposition}[performance of $\INTERFACE$]\label{p:interface}
    Under assumption \eqref{a:initial-grid} for the initial subdivision $\mathcal T_0$, given a refinement $\mathcal T$ of $\mathcal T_0$, let $\widetilde{\mathcal T}=\emph{\INTERFACE}(\mathcal T,r)$ with $r<2\diam(\mathcal G(\gamma,\mathcal T_0))$. \reviewred{Then there exists a positive constant
        $
            I_0:=I_0(\constShapeRegular,\gamma,\constComplexity)
        $
        so that}
    \begin{equation}\label{i:target-interface}
        \#(\widetilde{\mathcal T}) - \#(\mathcal T) \le I_0 r^{1-d} .
    \end{equation}
\end{proposition}
The proof is based on counting the number of bisections of $T\in \mathcal
    G(\gamma, \mathcal T_0)$. \reviewblue{Here we skip the proof and refer to the
    Appendix for more details.}

\begin{remark}
    \reviewblue{The above estimate holds provided that the initial refinement
        $\mathcal T_0$ is capable of capturing the shape of $\gamma$, i.e., that the
        assumption provided in Equation~\eqref{a:initial-grid} is valid.}
\end{remark}

A direct application of Proposition~\ref{p:interface} is to bound the cardinality of refined cells from $\INTERFACE$ in $\REGSOLVE$.
\begin{corollary}[performance of $\REGSOLVE\textrm{::}\INTERFACE$]\label{c:interface}
    Let $\{\widetilde{\mathcal T}_j\}$ be the sequence of subdivisions generated by $\emph{\INTERFACE}$ in $\emph{\REGSOLVE}$. Then at $j$-th iterate, there exists a positive constant
    $
        I_0:=I_0(\constShapeRegular,\gamma,f,\constComplexity)
    $
    satisfying
    \begin{equation}\label{i:error-regularization}
        \#(\widetilde{\mathcal T}_j) - \#(\mathcal T_j) \le I_0 \tau_j^{-2(d-1)} .
    \end{equation}
\end{corollary}
\begin{proof}
    The target estimate directly follows from \eqref{i:target-interface} by the relation $r\sim \tau^2$.
\end{proof}

\subsection{Performance of \texorpdfstring{$\SOLVE$}{SOLVE}}\label{ss:solve}
Let us review some estimates for the complexity of $\SOLVE$ following the analysis from \cite{cohen2012convergence}.

\subsubsection*{Contraction property}
One instrumental tool to evaluate the performance of $\SOLVE$ is the following contraction property \reviewblue{(cf. \cite[Theorem~4.3]{cohen2012convergence})}.

\begin{theorem}[contraction of $\SOLVE$]\label{t:solve-contraction}
    There exist two constants $\alpha\in (0,1)$ and $\widetilde\alpha>0$ depending on $\constShapeRegular$, $m$, $M$, and on the bulk parameter $\theta$ in $\emph{\SOLVE}$ such that for all $k\ge 0$,
    \[
        \vertiii{w_g-W_{k+1}}^2 + \widetilde\alpha\mathcal E(W_{k+1},\mathcal T_{k+1})^2
        \le \alpha^2 \reviewred{\Big(\vertiii{w_g-W_k}^2 + \widetilde\alpha \mathcal E(W_{k},\mathcal T_{k})^2\Big)} .
    \]
\end{theorem}

\subsubsection*{Approximation classes}
We denote $\mathscr T_n$ the set of all conforming subdivisions generated from $\mathcal T_0$ satisfying $\#(\mathcal T)\le n$. Define the best error obtained in $\mathscr T_n$
\[
    \sigma_n(u)_{\Hunz}:=\inf_{\mathcal T\in \mathscr T_n}\vertiii{u-U_{\mathcal T}}
\]
with $U_{\mathcal T}\in \mathbb V(\mathcal T)$ denoting the Galerkin projection of $u$, \ie
\[
    A(U_{\mathcal T}, V) = \langle F, V\rangle_{H^{-1}(\Omega),\Hunz},
    \quad\forall V\in \mathbb V(\mathcal T)
\]
and it also satisfies that
\[
    \vertiii{u-U_{\mathcal T}} = \inf_{V\in \mathbb V(\mathcal T)} \vertiii{u-V} .
\]
Define the approximation class $\mathcal A^{s}$ with $s\in (0,\tfrac1d]$ to be the set of all $v\in\Hunz$ such that the following quasi-semi-norm
\[
    |v|_{\mathcal A^s} := \sup_{n\ge 1}\Big(n^s\sigma_n(v)_{\Hunz}\Big)
\]
is finite. \reviewblue{Due to the nonzero jump of the normal derivative of $u$ on $\gamma$ and according to the discussion from Section~10 of \cite{berrone2019optimal}}, the best possible convergence rate is given by \reviewred{$s=\tfrac{1}{2(d-1)}$}.

\subsubsection*{Performance of \texorpdfstring{$\DATA$}{DATA}}
The approximation class $\mathcal A^s$ provides the rate of convergence for the energy error $\vertiii{u-U_{\mathcal T}}$. Recalling that given $g\in L^2(\Omega)$, the total error $E(w_g,\mathcal T)$ defined in \eqref{e:total-error} consists of both the energy error and the data indicator. So we are also concerned with the rate of convergence for the data indicator $\mathcal D(g,\mathcal T)$. Here we assume that
\begin{assumption}\label{a:data}
    For $\tau>0$ and \reviewred{a fixed bulk parameter $\widetilde\theta\in (0,1)$}, set $\mathcal T^* = \DATA(\mathcal T, g, \tau,\widetilde\theta)$. Then for $s\in (0,\tfrac1d]$, there exists a positive constant $G_s$ (depending on $g$ and $\widetilde \theta$) satisfying
    \[
        \#(\mathcal T^*) - \#(\mathcal T) \le G_s \tau^{-1/s} .
    \]
\end{assumption}

\subsubsection*{Cardinality of refined cells in \texorpdfstring{$\SOLVE$}{SOLVE}}
In the routine $\SOLVE$, we need to estimate the cardinalities of $\mathcal M_k$ as well as the cells refined from $\DATA$. The latter comes from Assumption~\ref{a:data}. The estimate of the former requires the following bulk property (cf. \cite[Lemma~5.2]{cohen2012convergence}):

\begin{lemma}\label{l:bulk-property}
    Assume that the bulk parameter $\theta\in (0, \theta_*)$ with
    \begin{equation}\label{e:theta-condition}
        \theta_*  =\frac{1}{\constEff\sqrt{1+C_L^2}} .
    \end{equation}
    Let $\mathcal T^*$ be a refinement of $\mathcal T$ and denote $\mathcal R_{\mathcal T\to \mathcal T^*}$ the set all refined cells from $\mathcal T$ to $\mathcal T^*$. If $E(w_g,\mathcal T^*)\le \xi E(w_g,\mathcal T)$ with
    \begin{equation}\label{e:xi}
        \xi :=\sqrt{1-\frac{\theta^2}{\theta_*^2}} ,
    \end{equation}
    there holds $\mathcal E(W_g,\mathcal R_{\mathcal T\to \mathcal T^*}) \ge \theta \mathcal E(W_g,\mathcal T)$.
\end{lemma}

Using Assumption~\ref{a:data} and the above lemma, Lemma~5.3 of \cite{cohen2012convergence} implies that for each iterate $k$ in $\SOLVE$, we have
\begin{equation}\label{i:mk}
    \#(\mathcal M_k) \lesssim (|w_g|_{\mathcal A^s}+G_s)^{1/s} E(w_g,\mathcal T_k)^{-1/s} .
\end{equation}

\subsection{Performance of \texorpdfstring{$\REGSOLVE$}{REGSOLVE}}\label{sss:regsolve}
In this section, we shall adapt the results in the previous subsection to $\REGSOLVE$.
\subsubsection{Performance of \texorpdfstring{$\DATA$}{DATA} using \texorpdfstring{$F^r$}{}}
To show that Assumption~\ref{a:data} holds for $g\in L^2(\Omega)$ with $s=\tfrac1d$, starting from a conforming initial subdivision $\mathcal T_0$ \reviewred{and using a greedy algorithm (see Algorithm~\ref{alg:greedy}), we can find a refinement $\mathcal T$ of $\mathcal T_0$ so that the data indicator $\mathcal D(g,\mathcal T)$ is smaller than a target tolerance $\tau$.}

\begin{algorithm}
    \begin{algorithmic}
        \STATE $\mathcal T=\mathcal T_0$
        \WHILE{$\mathcal D(g,\mathcal T) > \tau$}
        \STATE $T = \text{argmax}\{d(g,T,\mathcal T)\}$;
        \STATE $\mathcal T = \REFINE(\mathcal T,\{T\})$;
        \ENDWHILE
        \RETURN $\mathcal T$;
    \end{algorithmic}
    \caption{$\mathcal T = \GREEDY(\mathcal T_0, g,\tau)$}
    \label{alg:greedy}
\end{algorithm}
According to \cite[Theorem~7.3]{cohen2012convergence}, there exists a positive constant $K$ depending only on the shape regularity constant $\constShapeRegular$ such that
\[
    \#(\mathcal T) - \#(\mathcal T_0) \le K\|g\|_{L^2(\Omega)}^2\tau^{-d}.
\]
The above result can be extended by replacing $\mathcal T_0$ with its refinement $\mathcal T$, \ie $\mathcal T^* = \GREEDY(\mathcal T, g,\tau)$, and there holds
\begin{equation}\label{i:greedy-old}
    \#(\mathcal T^*) - \#(\mathcal T) \le K\|g\|_{L^2(\Omega)}^2\tau^{-d}.
\end{equation}
This is because the marked cells in $\GREEDY(\mathcal T, g,\tau)$ are contained in those generated by $\GREEDY(\mathcal T_0, g,\tau)$; see \cite[Proposition~2]{bonito2013adaptive} for a detailed discussion. Hence, any $L^2(\Omega)$ function $g$ satisfies Assumption~\ref{a:data} with $s=\tfrac1d$ and \reviewred{$\|g\|_{L^2(\Omega)}^2 \sim G_{1/d}$}. When $g=F^r$ as defined in Remark~\ref{r:regF}, the constant $G_{1/d}$ may still depend on $r$ in an arbitrary refinement of $\mathcal T_0$. However, the refinement process in $\DATA$ is based on the subdivisions generated by $\INTERFACE$. So cells marked in $\GREEDY$ should be located in a neighborhood of a tubular extension of $\gamma$, whose width can be controlled by the regularization parameter $r$. In order to see the dependence of \eqref{i:greedy-old} on $r$, we modify the argument of Lemma~7.3 of \cite{cohen2012convergence} and a detailed proof is provided in the Appendix.

\begin{lemma}[approximation class for $F^r$]\label{l:greedy}
    Assume that $f\in L^\infty(\gamma)$ and $F^r$ is defined as in Remark~\ref{r:regF} for any $r>0$. Letting the initial subdivision $\mathcal T_0$ satisfy \eqref{a:initial-grid}, we define $\widetilde{\mathcal T}=\emph{\INTERFACE}(\mathcal T_0, r)$ with $r<\constBoundary$. For any $\tau>0$, the cardinality of refined cells in $\mathcal T^*=\emph{\GREEDY}(\widetilde{\mathcal T}, F^r,\tau)$ can be bounded by
    \[
        \#(\mathcal T^*) - \#(\widetilde{\mathcal T}) \le K_0 r^{1-d/2}\|f\|_{L^\infty(\gamma)}^d \tau^{-d} ,
    \]
    where the constant $K_0$ is independent of $r$ and $\tau$. This implies that Assumption~\ref{a:data} holds for $F^r$ with $s=\tfrac12$ and $G_{1/2} \sim r^{1-d/2}\|f\|_{L^\infty(\gamma)}^d$ when $\mathcal T=\widetilde{\mathcal T}$.
\end{lemma}
\begin{remark}\label{r:lemma-membership}
    Following the proof of \cite[Proposition~2]{bonito2013adaptive}, we can extend the results in Lemma~\ref{l:greedy} by replacing $\widetilde{\mathcal T}$ with any of its refinements. More precisely speaking, let $\mathcal T^+$ be any refinement of $\widetilde{\mathcal T}$, and $\mathcal T^* = \GREEDY(F^r,\mathcal T^+,\tau)$. Then,
    \[
        \#(\mathcal T^*) - \#(\mathcal T^+) \lesssim r^{1-d/2}\|f\|_{L^\infty(\gamma)}^d \tau^{-d} .
    \]
\end{remark}

\reviewblue{
    \begin{remark}\label{r:lp}
        An estimate similar to the one in Lemma~\ref{l:greedy} could be also
        obtained when the local data indicator in \GREEDY is replaced by the
        surrogate $L^p(\Omega)$ data indicator defined by (7.1) of
        \cite{cohen2012convergence}. Here $p=\tfrac{2d}{d+2}$ so that
        $L^p(\Omega)$ is on the same nonlinear Sobolev scale of
        $H^{-1}(\Omega)$. Note that $\|F^r\|_{L^p(\Omega)}\lesssim
            \|f\|_{L^\infty(\gamma)}r^{1/p-1}=\|f\|_{L^\infty(\gamma)}r^{1/d-1/2}$.
        Applying \cite[Lemma~7.3]{cohen2012convergence} directly we get
        \[
            \#(\mathcal T^*) - \#(\widetilde{\mathcal T})
            \lesssim \|F^r\|^d_{L^p(\Omega)} \tau^{-d} \lesssim r^{1-d/2}\|f\|_{L^\infty(\gamma)}^d \tau^{-d} .
        \]
    \end{remark}
    \begin{remark}
        We note that by treating $F^r$ as an $L^2(\Omega)$ data, Lemma~\ref{l:greedy} also reveals the dependency of $r$ for for the decay of the oscillation $\osc(F^r,\mathcal T)$.
    \end{remark}
}

\reviewblue{Now we are in a position to verify Assumption~\ref{a:data} when $g=F^r$. The proof follows \cite[Theorem~7.5]{cohen2012convergence} using a contraction property of $\mathcal D(F^r,\mathcal T)$, a bulk property, and Lemma~\ref{l:greedy}. Here we again omit the proof.}

\begin{corollary}[performance of $\DATA$]\label{c:data}
    Under the assumptions provided by Lemma~\ref{l:greedy}, Assumption~\ref{a:data} holds with $s=\tfrac1d$ and $g=F^r$ starting from $\widetilde{\mathcal T}=\emph{\INTERFACE}(\mathcal T_0, r)$. Precisely speaking, given a refinement $\mathcal T$ of $\widetilde{\mathcal T}$, \reviewred{let $\mathcal T^*$ be the output of $\emph{\DATA}(\mathcal T, F^r, \tau,\widetilde\theta)$ with a fixed $\widetilde\theta\in (0,1)$. Then, there exists a constant $K_0>0$ not depending on $r$ or $\tau$ (but depending on $\widetilde \theta$) satisfying}
    \[
        \#(\mathcal T^*) - \#(\mathcal T) \le K_0 r^{1-d/2}\tau^{-d} .
    \]
\end{corollary}

\subsubsection{Quasi-monotoniciy of the data indicator}
The following lemma provides a quasi-monotonicity of $\mathcal D(F^r,\mathcal
    T)$ with respect to $r$. We note that this property relies on some additional
hypothesis on the forcing data $f$ and on the nonnegativity of $\delta^r$.
\begin{lemma}\label{l:quasi-monotonicity}
    Given $r_2<r_1$, let $\mathcal T$ be a refinement of $\emph{\INTERFACE}(\mathcal T_0, r_2)$. Then there holds that
    \[
        \mathcal D(F^{r_2}, \mathcal T) \lesssim \widetilde \beta^{d}
        \mathcal D(F^{r_1}, \mathcal T) + r_2 ,
    \]
    where $\widetilde \beta=\tfrac{r_2}{r_1}<1$.
\end{lemma}
\begin{proof}
    \reviewblue{We investigate the local data indicator for $F^{r_2}$ when i)
        $T$ is away from the tubular neighborhood of $\gamma$ with radius $r_2$,
        ii) $T$ intersects the tubular neighborhood and $f$ changes sign in
        $T$, and iii) $T$ intersects with the tubular neighborhood and $f$ is
        non-negative/non-positive. Clearly, $d(F^{r_2},T,\mathcal T) = 0$ when
        $\dist(T,\gamma) > r_2$. We shall focus on the other cases.}

    We recall from the configuration of $f$ in Section~\ref{s:prelim} that the set $I$ \reviewblue{defined in \eqref{e:sign-change}} separates the sign of $f$ in $\gamma$. Define
    \[
        \mathcal B:= \{T\in \mathcal T : T\cap B_{r_2}(x_0) \neq \emptyset \text{ for some } x_0\in I\}.
    \]
    \reviewblue{Since $h_T\lesssim r_2$ for $T\in\mathcal B$, there holds
        \begin{equation}\label{i:measure-sum}
            \sum_{T\in \mathcal B}|T|\lesssim r_2^{2} .
        \end{equation}
        Here the hidden constant above depends on the measure of $I$ in
        co-dimension 2.} Now we bound $d(F^{r_2},T ,\mathcal T)$. If $T\notin
        \mathcal B$, since $\delta^r$ is nonnegative, \reviewblue{and thanks to
        Assumption~\ref{a:app-dirac}(4),} we have $\delta^{r_2} \le \widetilde
        \beta^{d} \delta^{r_1}$. Hence,
    \[
        d(F^{r_2}, T,\mathcal T) \le \widetilde \beta^{d} d(F^{r_1}, T, \mathcal T).
    \]
    If $T\in \mathcal B$, there holds
    \[
        \reviewblue{d(F^{r_2}, T, \mathcal T)^2  \lesssim \frac{h_T^2}{r_2^{2d}} \int_T|B_{r_2}(x)\cap \gamma|^2 \diff x  \lesssim \frac{h_T^2}{r_2^{2d}} r_2^{2(d-1)} |T|\lesssim |T|} .
    \]
    By summing up all contributions above and invoking \eqref{i:measure-sum}, we arrive at
    \[
        \begin{aligned}
            \mathcal D(F^{r_2},\mathcal T)^2 & = \sum_{T\in \mathcal B} d(F^{r_2}, T, \mathcal T)^2 + \sum_{T\notin \mathcal B} d(F^{r_2}, T, \mathcal T)^2                                                                                            \\
                                             & \reviewblue{\lesssim \sum_{T\in \mathcal B} |T|} + \sum_{T\notin \mathcal B} \widetilde \beta^{2d} d(F^{r_1}, T, \mathcal T)^2 \lesssim r_2^2 + \widetilde \beta^{2d} \mathcal D(F^{r_1}, \mathcal T) ,
        \end{aligned}
    \]
    which concludes the proof.
\end{proof}
\begin{remark}
    If $f$ is nonnegative or non-positive along $\gamma$, according to the proof of Lemma~\ref{l:quasi-monotonicity}, we immediately get
    $
        \mathcal D(F^{r_2}, \mathcal T) \le \widetilde \beta^{d}
        \mathcal D(F^{r_1}, \mathcal T) .
    $
\end{remark}

\subsubsection{Performace of each subroutine in \texorpdfstring{$\REGSOLVE$}{}}
In terms of the approximation class for $\uepsilon$, Lemma~3.2 of \cite{bonito2013adaptive} enlightens us to exploit the fact that $\uepsilon$ is an approximation of $u$ and then to characterize approximation properties of $\uepsilon$ with the approximation class of $u$, \ie using the quasi-semi-norm $|u|_{\mathcal A^s}$ for some $s\in (0,\tfrac1d)$.

\begin{lemma}[Lemma~3.2 of \cite{bonito2013adaptive}]\label{l:epsilon-approximation}
    If $\vertiii{u-\uepsilon} < \varepsilon$ for some $\varepsilon >0$, then $\uepsilon$ is a \emph{$2\varepsilon$-approximation} to $u$ of order $s$: for all $\delta>2\varepsilon$, there exists a positive integer $n$ such that
    \[
        \sigma_n(\uepsilon)_{\Hunz} \le \delta, \quad\text{ and }\quad
        n\lesssim |u|_{\mathcal A^s}^{1/s}\delta^{-1/s} .
    \]
\end{lemma}

\begin{lemma}[a priori asymptotic decay of the total error, see Lemma~5.1 of
        \cite{cohen2012convergence}]\label{l:a-priori-deday} Under the settings in
    Lemma~\ref{l:greedy}, we set $r=r(\tau)$ according to
    \eqref{i:choice-regularization} so that $\vertiii{u-u^r}\le \constRel\tau/2$
    for some $\tau>0$.  Then for any $1 > \delta\ge \sqrt2\constRel\tau$, there
    is a refinement $\mathcal T$ of $\widetilde{\mathcal
            T}=\emph{\INTERFACE}(\mathcal T_0, r)$ such that
    \[
        E(\uepsilon,\mathcal T) \le \delta \quad\text{and}\quad
        \#(\mathcal T) - \#(\widetilde{\mathcal T}) \lesssim  (K_0 r^{1-d/2} + |u|_{\mathcal A^s}^{1/s}) \delta^{-1/s}.
    \]
\end{lemma}
\begin{proof}
    \reviewblue{A desired refinement $\mathcal T$ of $\widetilde{\mathcal T}$ is the overlay of $\mathcal T_f^r=\DATA(\widetilde{\mathcal T},F^r, \delta/\sqrt2)$ and $\mathcal T_\uepsilon$ from Lemma~\ref{l:epsilon-approximation} by replacing $\delta$ with $\frac{\delta}{\sqrt{2}}$.}
\end{proof}

\reviewblue{The next lemma provides the estimate of marked cells in $\SOLVE$. The proof follows from \cite[Lemma~5.3]{cohen2012convergence}, together with Lemma~\ref{l:a-priori-deday}, as well as the minimal assumption of $\texttt{MARK}$.}
\begin{lemma}[cardinality of $\REGSOLVE\textrm{::}\SOLVE\textrm{::}\MARK$]\label{l:marked-cells-refine}
    Under the settings given by Lemma~\ref{l:greedy}, let the bulk parameter $\theta$ defined in $\emph{\SOLVE}$ satisfy the condition $\theta < \theta_*$, with $\theta_*$ provided by \eqref{e:theta-condition}. For a fixed $\tau>0$, set $r=r(\tau)$ in \eqref{i:choice-regularization} and $\widetilde{\mathcal T}=\emph{\INTERFACE}(\mathcal T_0, r)$. We also let $\{\mathcal T_k\}$ be defined in $\emph{\SOLVE}(\widetilde{\mathcal T}, F^r, \widetilde \mu\tau)$ with $\widetilde \mu \ge \sqrt2\constRel/(\xi\constEff)$ and $\{\mathcal M_k\}$ be the set of marked cells generated from $\emph{\SOLVE\textrm{::}\MARK}$ at $\mathcal T_k$. Then there holds
    \[
        \#(\mathcal M_k) \lesssim (K_0 r^{1-d/2}+ U_s) E(\uepsilon, \mathcal T_k)^{-1/s} ,
    \]
    where $U_s := |u|_{\mathcal A^s}^{1/s}$.
\end{lemma}

\begin{lemma}[\reviewblue{performance} of $\REGSOLVE\textrm{::}\SOLVE$ (cf. Theorem~4.1 of \cite{bonito2013adaptive})]\label{l:marked-cell-solve}
    Denote $\{(\mathcal T_j, U_j)\}_{j=0}^{j_{\max}}$ to be the sequence of subdivisions and approximations of $u$ generated by $\emph{\REGSOLVE}$, respectively. Set $\widetilde{\mathcal T}_j = \emph{\INTERFACE}(\mathcal T_j, r_j)$ with $r_j=r(\tau_j)$. Under the assumptions provided by Lemma~\ref{l:greedy} and Lemma~\ref{l:marked-cells-refine}, there holds that for $j\ge 1$,
    \[
        \#(\mathcal T_{j+1}) - \#(\widetilde{\mathcal T}_j) \lesssim (K_0 r^{1-d/2}+ U_s) \tau_j^{-1/s}.
    \]
\end{lemma}

\begin{proof}
    For each $j\ge 1$, we let $k_{\max}$ be the number of iterations executed in $\SOLVE$. Let us first show that $k_{\max}$ is uniform bound with respect to $j$. Let $\widehat \tau_j$ be the error indicator for $\widetilde U_j = \GAL(\widetilde{\mathcal T}_j, F^{r_j})$ with $r_j=r(\tau_j)$ in $\REGSOLVE$. In view of \eqref{i:upper-lower-bound} and Lemma~\ref{l:quasi-monotonicity}, we have
    \[
        \begin{aligned}
            \widehat{\tau}_j \lesssim E(u^{r_j}, \widetilde{\mathcal T}_j)
             & \lesssim \vertiii{u^{r_j}-\widetilde U_j} + \mathcal D(F^{r_j}, \widetilde{\mathcal T}_j )
            \lesssim \vertiii{u^{r_j}- U_j} + \mathcal D(F^{r_{j-1}}, \widetilde {\mathcal T_j}) + r_j         \\
             & \lesssim \vertiii{u- u^{r_j}} + \vertiii{u- U_j} + \mathcal E(u^{r_{j-1}}, \mathcal T_j) + r_j.
        \end{aligned}
    \]
    Now we invoke Proposition~\ref{p:error-control} and \ref{p:reg-error-sol} to deduce
    \begin{equation}\label{i:error-error}
        \widehat{\tau}_j \lesssim r_j^{1/2} + \tau_{j-1} + r_j
        \lesssim \tau_{j} + \tau_{j-1} \lesssim \tau_{j} .
    \end{equation}
    In the above estimates we also used the relations $r_j\lesssim \tau_{j+1}^2$ and $\tau_{j} = \beta\tau_{j-1}$. The contraction property \eqref{t:solve-contraction} together with \eqref{i:error-error} yields the uniform boundedness of $k_{\max}$.

    \reviewblue{At each iteration $k=0,1,\ldots,k_{\max}$ in $\SOLVE$, Lemma~\ref{l:marked-cells-refine} controls the number of marked cells in $\REFINE$. For the cardinality of the marked cells in $\DATA$, we set $\mathcal T^+_k$ to be the corresponding output and apply Corollary~\ref{c:data} to get,
    \[
        \#(\mathcal T_k^+) - \#(\mathcal T_k) \lesssim K_0 r^{1-d/2}(\lambda\theta\mathcal E_k)^{-1/s}
        \lesssim K_0 r^{1-d/2}E(\uepsilon,\mathcal T_k)^{-1/s} .
    \]
    Combining the above estimate together with Lemma~\ref{l:marked-cells-refine}, we obtain that
    \begin{equation}\label{i:bound-rate}
        \begin{aligned}
            \#(\mathcal T_{j+1}) - \#(\widetilde{\mathcal T_j}) & \lesssim
            \sum_{k=0}^{k_{\max}} \big(\#(\mathcal M_k) + \#(\mathcal T_k^+) - \#(\mathcal T_k)\big)          \\
                                                                & \lesssim (K_0 r^{1-d/2} + U_s)
            E(\uepsilon,\mathcal T_{k_{\max}})^{-1/s} \sum_{k=0}^{k_{\max}} \alpha^{(k_{\max}-k)/s}           \\
                                                                & \lesssim (K_0 r^{1-d/2}+ U_s)  \tau^{-1/s},
        \end{aligned}
    \end{equation}
    where for the last two inequalities above we applied Theorem~\ref{t:solve-contraction}, $\tau \lesssim E(\uepsilon,\mathcal T_{k_{\max}})$, and $\sum_{k=0}^{k_{\max}} \alpha^{(k_{\max}-k)/s} \le \sum_{k=0}^\infty \alpha^{k/s}\lesssim 1$. The proof is complete.}
\end{proof}

\subsubsection{Performace of \texorpdfstring{$\REGSOLVE$}{REGSOLVE}}
We are now in a position to show our main result.

\begin{theorem}[\reviewblue{performance} of $\REGSOLVE$]\label{t:convergence-rate}
    Denote $\{(\mathcal T_j, U_j)\}_{j=0}^{j_{\max}}$ to be the sequence of subdivisions and approximations of $u$ generated by $\emph{\REGSOLVE}$, respectively. Under the assumptions provided by Lemma~\ref{l:greedy} and Lemma~\ref{l:marked-cells-refine}, there holds that
    \[
        \#(\mathcal T_{j_{\max}}) - \#(\mathcal T_0) \lesssim (K_0+I_0+U_s) \tau_{j_{\max}}^{2-d-1/s} .
    \]
\end{theorem}
\begin{proof}
    Denote $\mathcal M_j$ the collections of cells marked for refinement in the $j$-th iteration of solve. Invoking Corollary~\ref{c:interface} and Lemma~\ref{l:marked-cell-solve}, we have
    \[
        \begin{aligned}
            \#(\mathcal M_j) & = (\#(\widetilde{\mathcal T}_j) - \#(\mathcal T_j)) + (\#(\mathcal T_{j+1}) - \#(\widetilde{\mathcal T}_j)) \\
                             & \lesssim (I_0 \tau_j^{2-d} + K_0 \tau_j^{2-d}  + U_s) \tau_j^{-1/s}
            \lesssim (I_0 + K_0  + U_s) \tau_j^{2-d-1/s},
        \end{aligned}
    \]
    where we used the setting $r \sim \tau_j^2$ according to \eqref{i:choice-regularization}. \reviewblue{Summing up the above estimate for $j=0,\ldots,j_{\max}-1$ together with the relation $\tau_{j_{\max}} = \beta^{j_{\max}-j} \tau_j$ implies the target estimate.}
\end{proof}

\begin{remark}[convergence rates]\label{r:convergence-rates}
    \reviewblue{Since the best possible rate is $s=\tfrac{1}{2(d-1)}$, Theorem~\ref{t:convergence-rate} implies that
        \[
            \vertiii{u-U_j} \lesssim\tau_j\lesssim (\#(\mathcal T_j)-\#(\mathcal T_0))^{-1/(3d-4)} .
        \]
        Hence, in two dimensional space, we guarantee that the adaptive method is quasi-optimal. However, in three dimensional space, we have,
        \[
            \vertiii{u-U_j} \lesssim (\#(\mathcal T_j)-\#(\mathcal T_0))^{-1/5},
        \]
        which turns out to be sub-optimal compared with the optimal rate $\tfrac14$.}
\end{remark}

\section{Numerical illustration}\label{s:numerics}
In this section, we test our numerical algorithm proposed in Section~\ref{s:algorithm} for the following interface problem: letting $\gamma$ be defined as in \eqref{i:constBoundary}, we want to find $u$ satisfying
\begin{equation}\label{e:test}
    \begin{aligned}
        -\Delta u = 0, \quad                   & \text{ in } \Omega\backslash\gamma, \\
        [u] \ = 0, \quad                       & \text{ on } \gamma,                 \\
        [\GRAD u \cdot \nu_\gamma ] = f, \quad & \text{ on } \gamma,                 \\
        u = g, \quad                           & \text{ on } \partial\Omega ,
    \end{aligned}
\end{equation}
where $[.]$ denotes the jump of the function across the interface $\gamma$ and $\nu_\gamma$ is the outward normal direction along $\gamma$. So $u$ satisfies the weak formulation \eqref{e:variational} with the forcing data $F$ defined by \eqref{e:dirac} and a non-homogeneous boundary condition.

\reviewblue{As we mentioned in Section~\ref{sss:refine}, our numerical
    implementation relies on the \texttt{deal.II} finite element library
    \cite{dealII93,dealiidesign} and we use quadrilateral subdivisions in two
    dimensions and hexahedral subdivisions in three dimensions. For the
    computation of the right hand side of the discrete system, we refer to
    Remark~22 of \cite{heltai2020priori} for more details. In the following
    numerical simulations, we use a radially symmetric $C^1$ approximation of
    the Dirac delta approximation, \ie $\psi_\rho(x) = c_d(1+\cos(|\pi
        x|))\chi(x)$, where $\chi(x)$ is the characteristic function on the unit
    ball and $c_d$ is a normalization constant so that $\int_{\Rd}\psi_\rho =
        1$.}

In $\REGSOLVE$, we fix $\widetilde \mu =\tfrac12$. The parameters $\mathcal T_0$
(initial subdivision), $\tau_0$ (initial tolerance), $\beta$ (tolerance
reduction), $j_{\max}$ (number of iterations), the bulk parameters $\theta$ and $\widetilde\theta$, and $\lambda$ (ratio between $\mathcal E$ and $\theta \mathcal D$) will be
provided for each numerical test. For the regularization parameter, we simply
set $r(\tau_j)=\tau_j^2$ in $\REGSOLVE\textrm{::}\INTERFACE$ to avoid the
estimate of the constants $\constReg$, $\constRel$ and $\|f\|_{L^2(\gamma)}$ in
\eqref{i:choice-regularization}. Furthermore, after the last iteration of
$\REGSOLVE$, we perform the following extra steps
\begin{algorithmic}
    \STATE $r_{j+1} = r(\tau_{j_{\max}+1})$;
    \STATE $\widetilde {\mathcal T}_{j_{\max}+1} = \INTERFACE(\mathcal T_{j_{\max}+1}, r_{j+1})$;
    \STATE $\GAL(\widetilde{\mathcal T}_{j_{\max}+1}, F^{r_{j+1}})$;
\end{algorithmic}

\subsection{Convergence tests on a L-shaped domain}
Following similar test cases to those presented in~\cite{HeltaiRotundo-2019-a}, we set $\Omega=(-1,1)^2\backslash [0,1]^2$, $\gamma=\partial B_{R}(c)$ with $R=0.2$ and $c=(0.5,-0.5)^{\mathtt T}$, $f=\tfrac{1}{R}$ and $g=\ln(|x-c|)$. The analytic solution is given by
\[
    u(x)=r(x)^{2/3}\sin(\tfrac23(\theta(x)-\tfrac\pi2))+
    \left\{
    \begin{aligned}
        -\ln(|x-c|), & \quad \text{if } |x-c| >R,    \\
        -\ln(R),     & \quad \text{if } |x-c|\le R ,
    \end{aligned}
    \right .
\]
with $(r,\theta)$ denoting the polar coordinates. We start with an initial uniform grid $\mathcal T_0$ with the mesh size
$\sqrt{2}/4$. Note that we also approximate the interface $\gamma$ with a
uniform subdivision whose vertices lie on $\gamma$. The corresponding mesh size
is fixed as $2\pi R/2^{14}$ so that the geometric error will not dominate the
total error. For the parameters showing the numerical algorithm, we set
$j_{\max}=6$, $\tau_0=0.6$, $\beta=0.8$, $\lambda=\tfrac13$ and
$\theta=\widetilde\theta = 0.7$ in $\SOLVE$ and $\DATA$, respectively. The left
plot in Figure~\ref{f:lshaped-error} reports the $H^1(\Omega)$-error versus the
number of degrees of freedom (\#DoFs) when $\GAL$ is executed. We note that the
error goes down almost vertically when we update the regularization radius after
$\INTERFACE$. In order to verify Theorem~\ref{t:convergence-rate} (or
Remark~\ref{r:convergence-rates}), we extract the sampling points only for $U_j$
(\ie the last Galerkin approximation in each iteration of $\REGSOLVE$) in red. Based on the observation we confirm the first order rate of convergence. We also present our
approximated solution $U_3$ and its underlying subdivision in Figure~\ref{f:lshaped-subdivision-and-solution}.

We test the algorithm in Remark~\ref{r:another-regsolve} (i.e., we make one
single iteration, and set the initial target tolerance to
$\tau_0\beta^{j_{\max}}$), and report the energy error for the final
approximation against \#DoFs in the right plot of Figure~\ref{f:lshaped-error}. Here we
use the same parameters except that $j_{\max}=14$, in order to reach a similar
true error. Comparing with the left plot of Figure~\ref{f:lshaped-error}, we
note that although both algorithms guarantee the quasi-optimal convergence rate,
the energy error $\vertiii{U_{j_{\max}}-u}$ using the algorithm in
Remark~\ref{r:another-regsolve} is much larger than that computed from
$\REGSOLVE$ with multiple iterations.

\begin{figure}[hbt!]
    \begin{center}
        \begin{tabular}{cc}
            \includegraphics[scale=0.5]{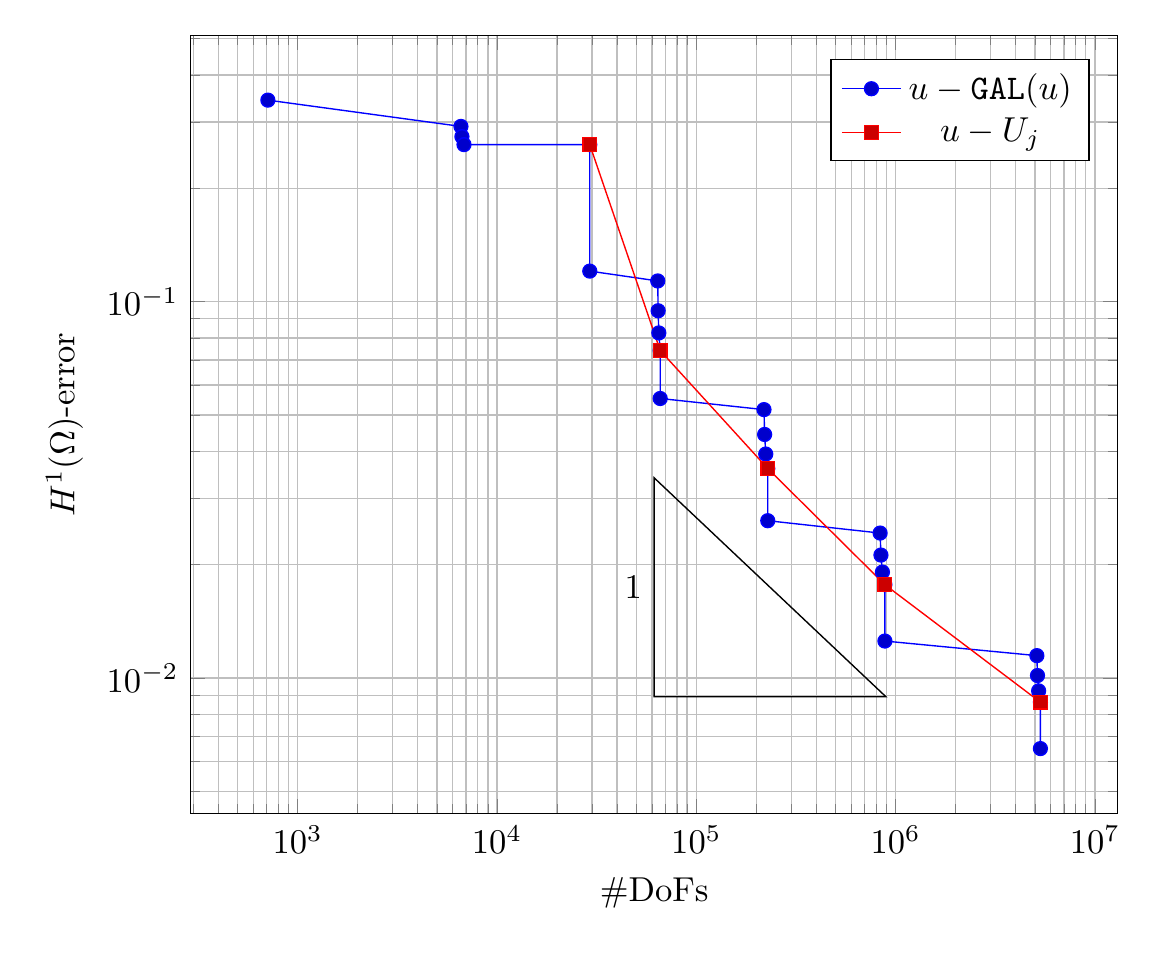}
             & \includegraphics[scale=0.5]{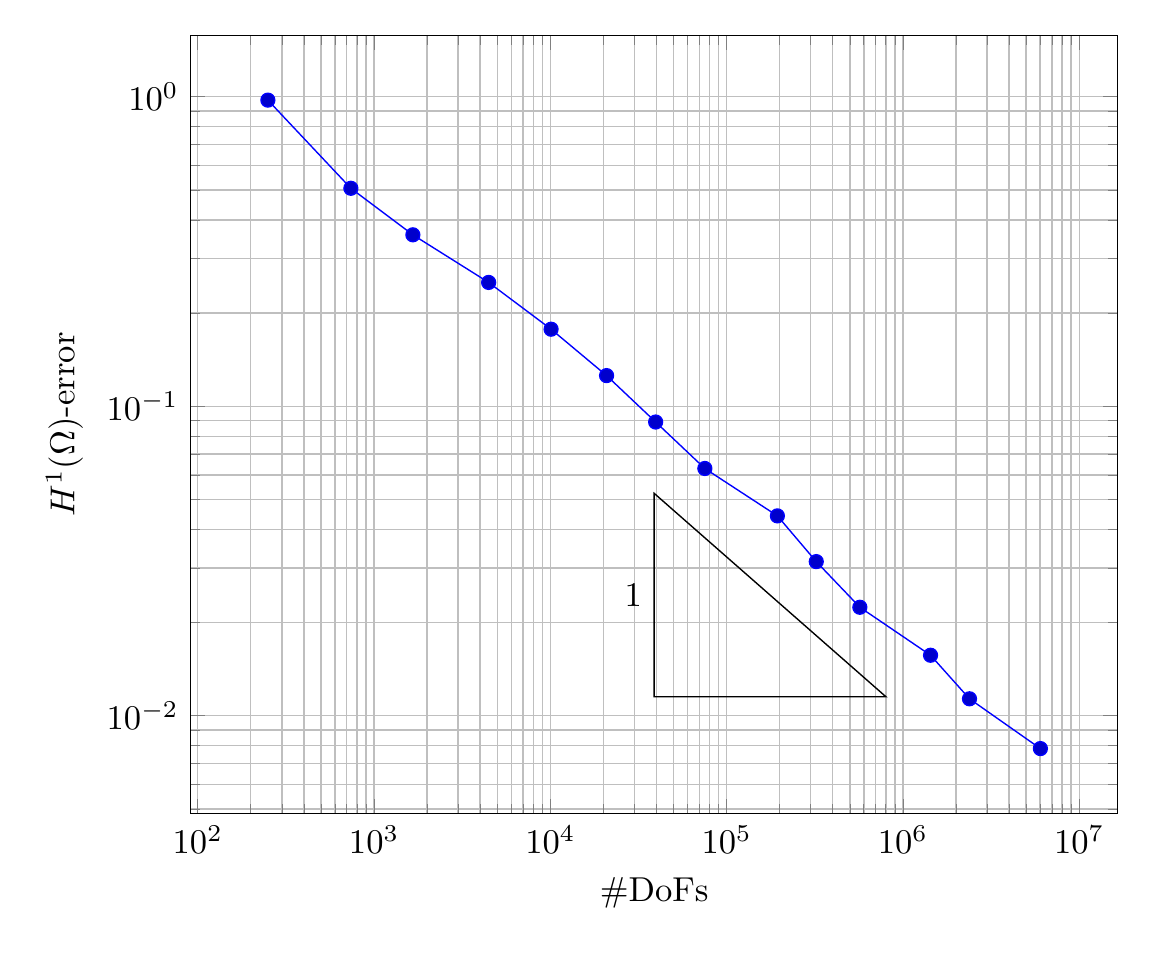} \\
        \end{tabular}
    \end{center}
    \caption{Test on a L-shaped domain: (left) $H^1(\Omega)$-error decay between the
        solution $u$ and every Galerkin approximation (\emph{$\GAL(u)$}) in
        $\emph{\REGSOLVE}$  and between $u$ and $U_j$ defined in $\emph{\REGSOLVE}$,
        (right) $H^1(\Omega)$-error decay between $u$ and $U_{j_{\max}}$ defined
        from Remark~\ref{r:another-regsolve}.  We set
        $j_{\max}=6$, and $\tau_0 = .6$. }
    \label{f:lshaped-error}
\end{figure}

\begin{figure}[hbt!]
    \begin{center}
        \begin{tabular}{cc}
            \includegraphics[scale=0.08]{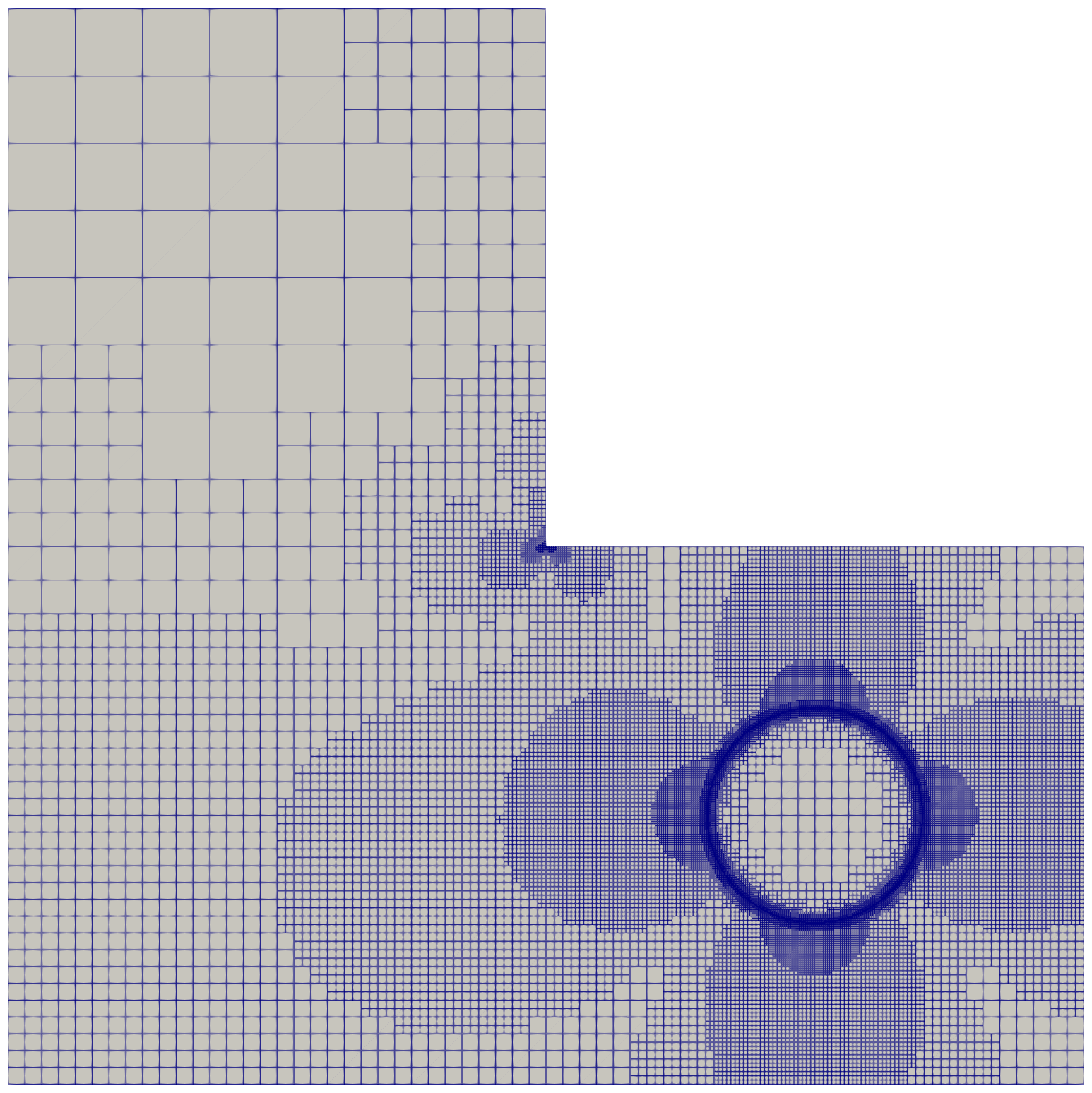}
             & \includegraphics[scale=0.08]{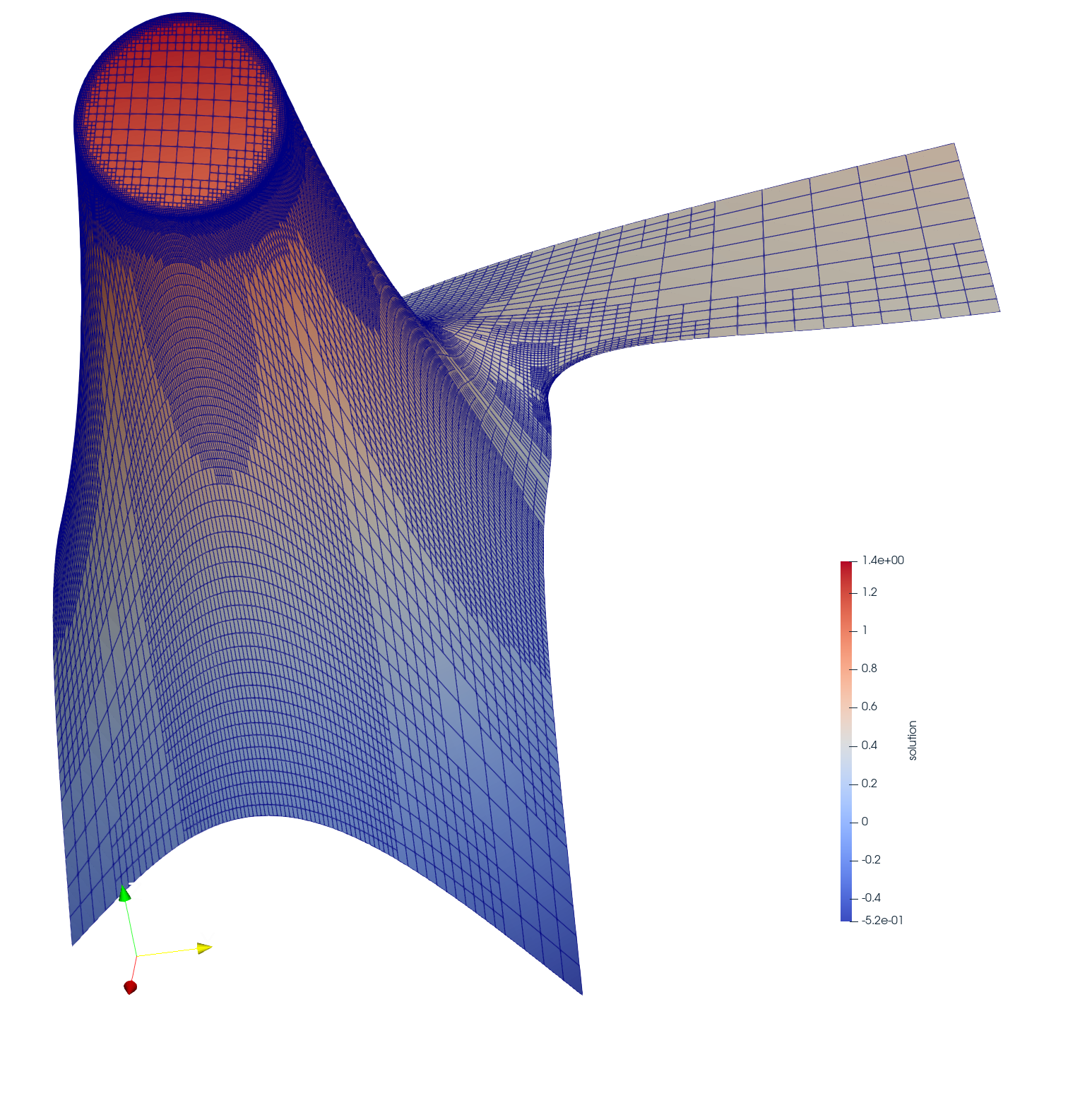} \\
        \end{tabular}
    \end{center}
    \caption{Test on a L-shaped domain: (left) the subdivisions of $U_3$ in \emph{$\REGSOLVE$} and (right) the corresponding Galerkin approximation using \emph{Tensor product $C^1$}.}
    \label{f:lshaped-subdivision-and-solution}
\end{figure}

\subsection{Convergence tests in the unit cube}
We test our numerical algorithm in 3d by setting $\Omega=(-1,1)^3$ and $\gamma=\partial B_{R}(c)$ with $R=0.2$ and $c=(0.3,0.3,0.3)^{\mathtt T}$. We also set the data function $f=\tfrac{1}{R^2}$ on $\gamma$ and $g=1/|x-c|$ so that the analytic solution is given by
\[
    u(x) = \left\{
    \begin{aligned}
        1/{|x-c|}, & \quad \text{if } |x-c| >R,    \\
        1/{R},     & \quad \text{if } |x-c|\le R .
    \end{aligned}
    \right.
\]
\reviewblue{We start with an initial uniform grid $\mathcal T_0$ with the mesh size $\sqrt{3}/16$. To approximate the interface $\gamma$, we start with initial quasi-uniform coarse mesh and refine it globally 7 times so that the geometric error is small enough. For the other approximation parameters, we set $j_{\max}=5$, $\tau_0=1.5$, $\beta=0.8$, $\lambda=1$, $\theta=0.5$ in $\SOLVE$ and $\widetilde \theta = 0.8$ in $\DATA$. In Figure~\ref{f:cube-error}, we report the $H^1(\Omega)$-error against \#DoFs for the following three different types of $\delta^r$: the \emph{radially symmetric $C^1$} type, the \emph{tensor product $C^\infty$} type generated by $\psi_{1d}(x)=\exp(1-1/(1-x^2))\chi_{(-1,1)}(x)$, and the \emph{tensor product $L^\infty$} type generated by $\psi_{1d}(x)=\tfrac12\chi_{(-1,1)}(x)$. For each error plot, we also report the slope of the linear regression of the last five sampling points. For the choice of \emph{tensor product $C^\infty$}, the performance is sub-optimal and close to the predicted rate $\tfrac15$. When using \emph{radially symmetric $C^1$}, the observed convergence rate is better than the best possible rate $\tfrac14$. As for \emph{tensor product $L^\infty$}, the performance is between $\tfrac14$ and $\tfrac15$. We also report the coarse grid and the grids for $U_5$, $U_7$ as well as the approximation $U_7$ using \emph{radially symmetric} $C^1$ in Figure~\ref{f:cube-subdivision}.}

\begin{figure}[hbt!]
    \begin{center}
        \includegraphics[scale=0.5]{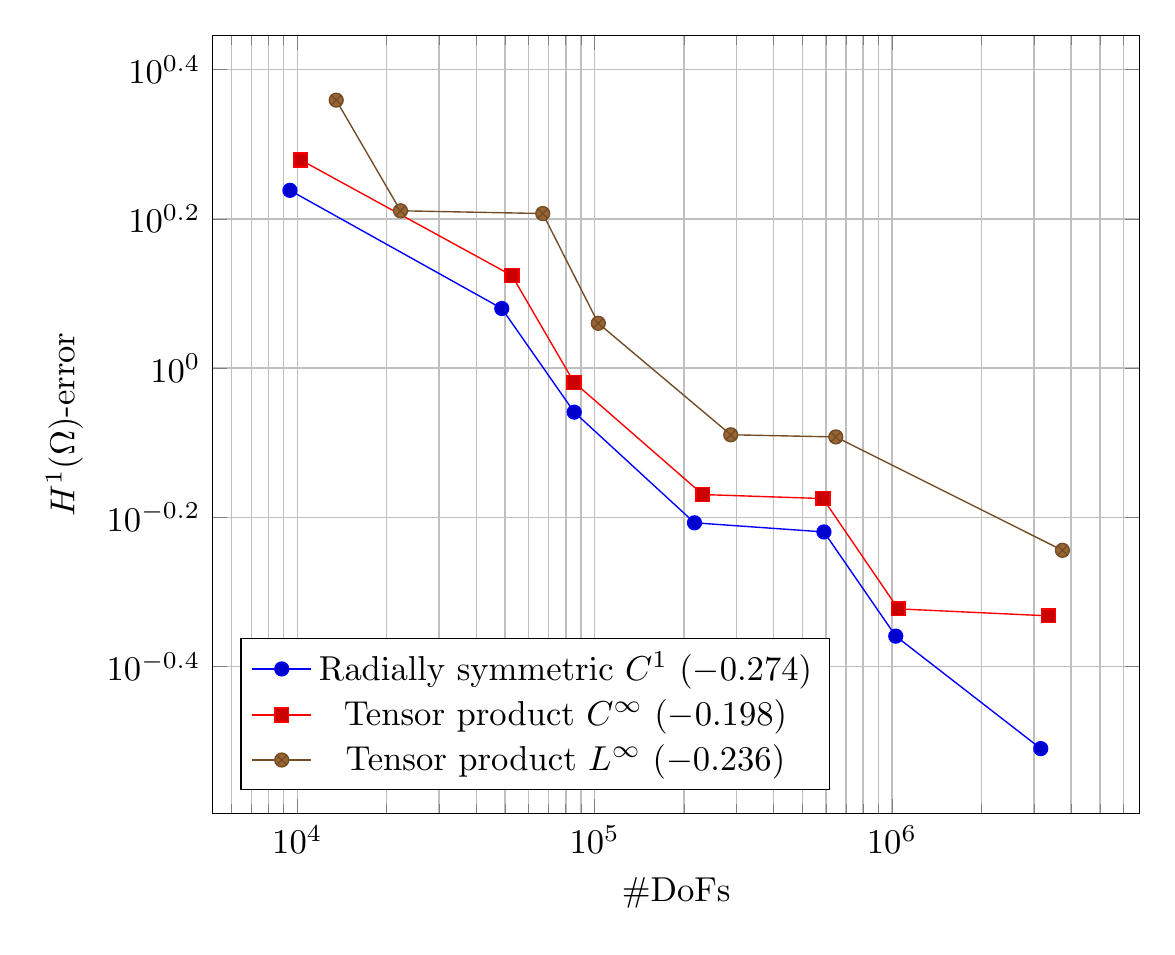}
    \end{center}
    \caption{Tests in the unit cube: $H^1(\Omega)$-error decay between the solution $u$ and $U_j$ for $j=1,\ldots,7$ and for different choices of $\depsilon(x)$. In terms of each plot, the slope of the linear regression of the last five sampling points is reported in the legend.}
    \label{f:cube-error}
\end{figure}

\begin{figure}[hbt!]
    \begin{center}
        \begin{tabular}{ccc}
            \includegraphics[scale=0.20]{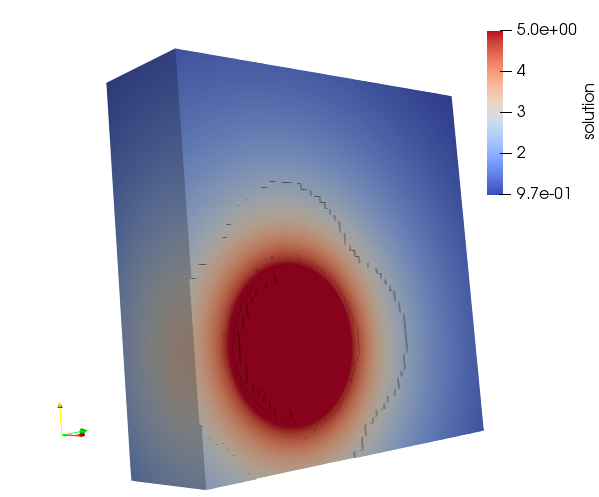} & \includegraphics[scale=0.20]{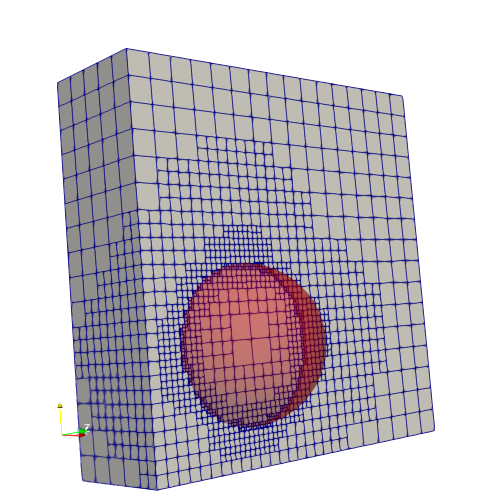} &
            \includegraphics[scale=0.20]{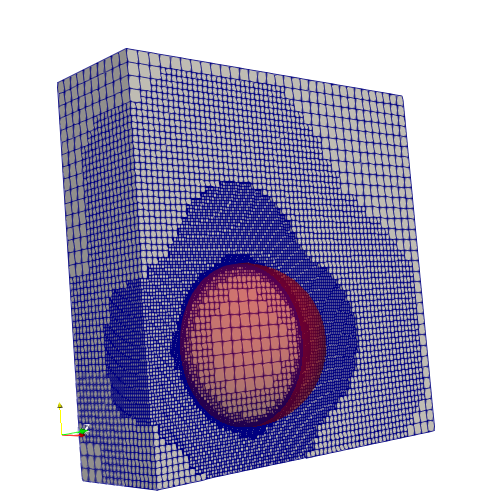}                                                 \\
        \end{tabular}
    \end{center}
    \caption{Tests in the unit cube: the crinkle clip ($x_1\le 0.3$) of the approximation $U_7$ ($3741904$ DoFs) (left) as well as the subdivisions for $U_5$ (mid) and $U_7$ (right) using \emph{radially symmetric} $C^1$. The interface $\gamma$ is marked in red.}
    \label{f:cube-subdivision}
\end{figure}

\subsection{Performance tests in the unit square}\label{ss:square}
Consider $\Omega=(0,1)^2$, $\gamma=\partial B_{R}(c)$ with $R=0.2$ and $c=(0.3,0.3)^{\Tr}$, $f=\tfrac{1}{R}$ and $g=\ln(|x-c|)$. Similar to the previous section, we can obtain the following exact solution
\[
    u(x) = \left\{
    \begin{aligned}
        -\ln(|x-c|), & \quad \text{if } |x-c| >R,    \\
        -\ln(R),     & \quad \text{if } |x-c|\le R .
    \end{aligned}
    \right.
\]
We shall compare the performance of our numerical algorithms both in Algorithm~\ref{alg:regsolve} and Remark~\ref{r:another-regsolve} with the algorithm without regularization; see the numerical algorithm from Section~7.2 of \cite{cohen2012convergence}. To be more precise, the algorithm without using the regularization is based on $\SOLVE$ by replacing the data indicator $\mathcal D$ with the following surrogate data indicator:
\[
    \widetilde {\mathcal D} (f, T,\mathcal T):= h_T^{1/2} \|f\|_{L^2(T\cap\gamma)} .
\]
Using the exact solution $u$, after the $j$-th iterate of $\REGSOLVE$ in Algorithm~\ref{alg:regsolve} using \emph{tensor product} $L^\infty$, we compute the $H^1(\Omega)$-error between $u$ and $U_j$, denoted by $e_j$. For the parameters we set $\tau_0=0.3$, $\beta=0.7$, $\lambda=\tfrac13$, $\theta=\widetilde\theta=0.55$. Then we run the non-regularized program with the same parameters and terminate it when the energy error is smaller than $e_j$, denoting $\widetilde e_j$ the energy error for the corresponding output approximation. For $j=8,9,\ldots,12$. We also compute $e_j$ using the algorithm provided by Remark~\ref{r:another-regsolve} with $j=11,\ldots,15$. Now we report those errors and the CPU times for each program against \#DoFs in Figure~\ref{f:square-error-and-time}. We observe that all algorithms are quasi-optimal but the algorithm from Remark~\ref{r:another-regsolve} requires more DoFs.

In terms of the computation time, it turns out that
Algorithm~\ref{alg:regsolve} needs more time when the discrete system is small
(less than $10^7$) and becomes more efficient when the size of the system is
increasing. \reviewblue{ Since the computational cost associated to the
    regularized version is comparable to the one required by the non-regularized
    version \emph{when computed on the same grid} (see \cite[Figure
        7]{heltai2020priori}), a fair comparison between the different AFEM algorithms
    should keep into account the computational cost in terms of the attained
    accuracy.}

\reviewblue{The regularized case reaches lower errors, for the same number of
    degrees of freedom, but it is more expensive (due to a larger number of refined
    elements around the interface $\gamma$ required by our algorithms). The
    computational cost is compensated for by the higher accuracy in the largest
    scale computations, where the computational cost per degree of freedom is
    comparable, making the regularized approach roughly comparable to the
    non-regularized one also in the AFEM context.} In Figure~\ref{f:square-grids},
we finally report the grid for $U_5$ using Algorithm~\ref{alg:regsolve} and
corresponding grid for the non-regularized algorithm.

\begin{figure}[hbt!]
    \begin{center}
        \begin{tabular}{ccc}
            \includegraphics[width=0.3\textwidth]{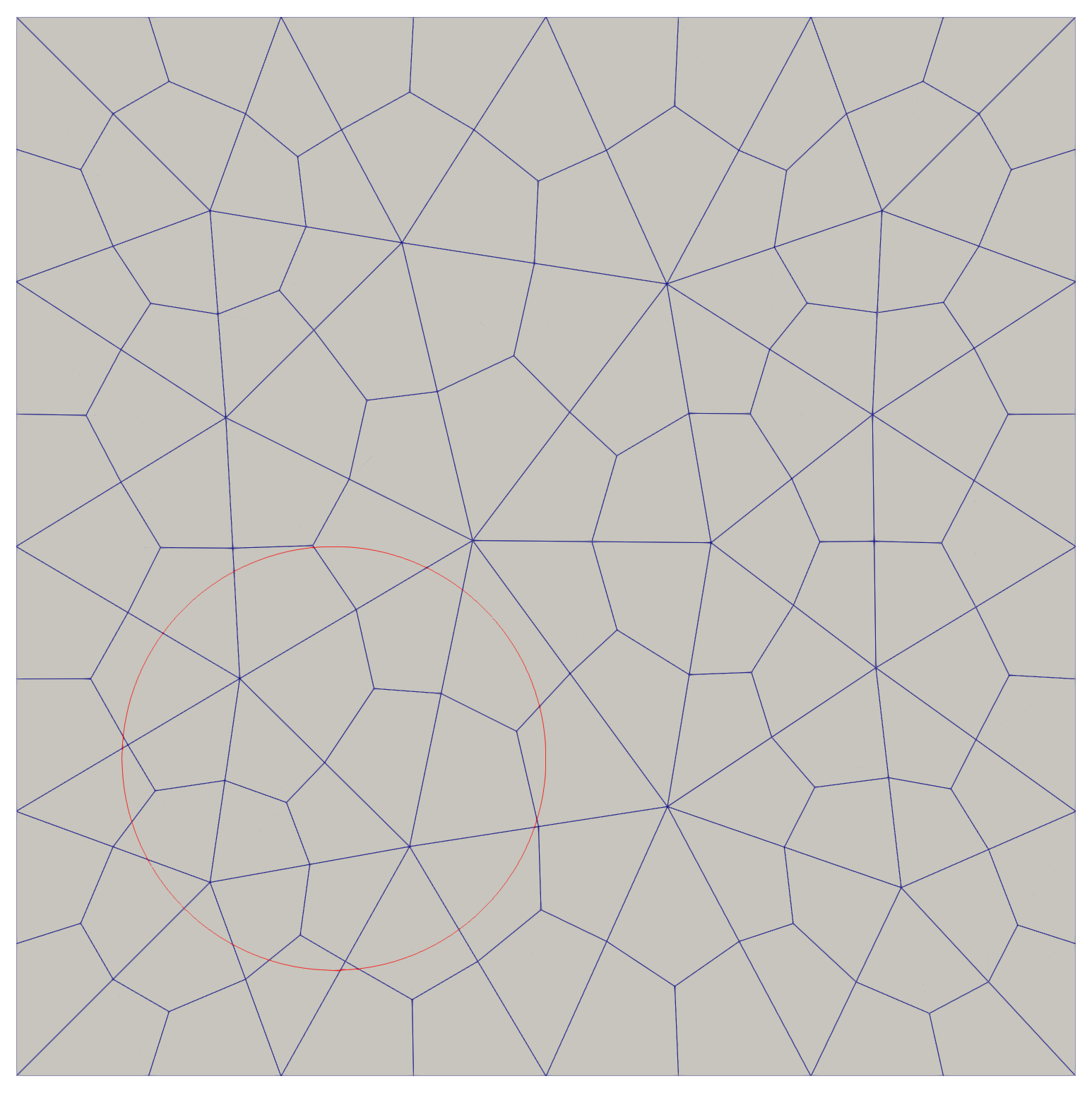} &
            \includegraphics[width=0.3\textwidth]{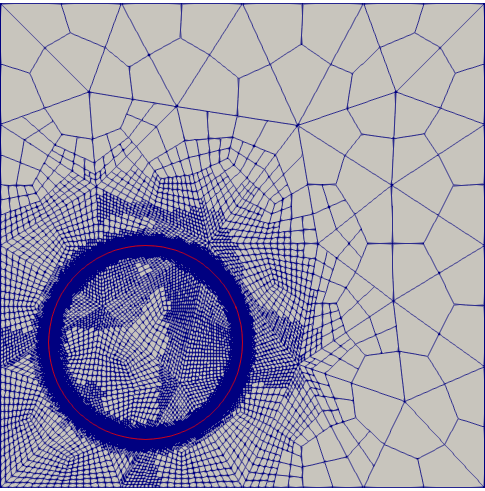}
                                                                                & \includegraphics[width=0.3\textwidth]{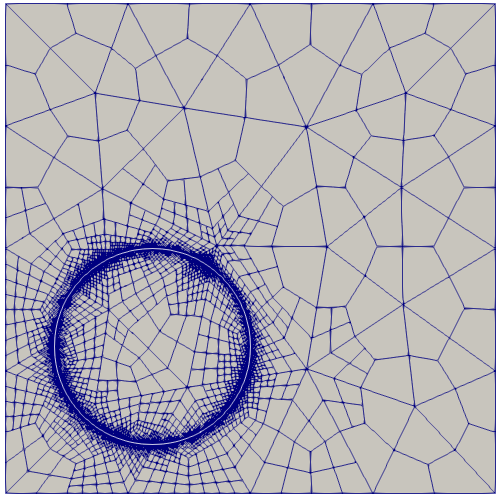} \\
        \end{tabular}
    \end{center}
    \caption{Tests on a square: (left) the unstructured coarse mesh $\mathcal T_0$, (center) the subdivision for $U_5$, and (right) the corresponding subdivision using non-regularized algorithm.}
    \label{f:square-grids}
\end{figure}

\begin{figure}[hbt!]
    \begin{center}
        \includegraphics[scale=0.5]{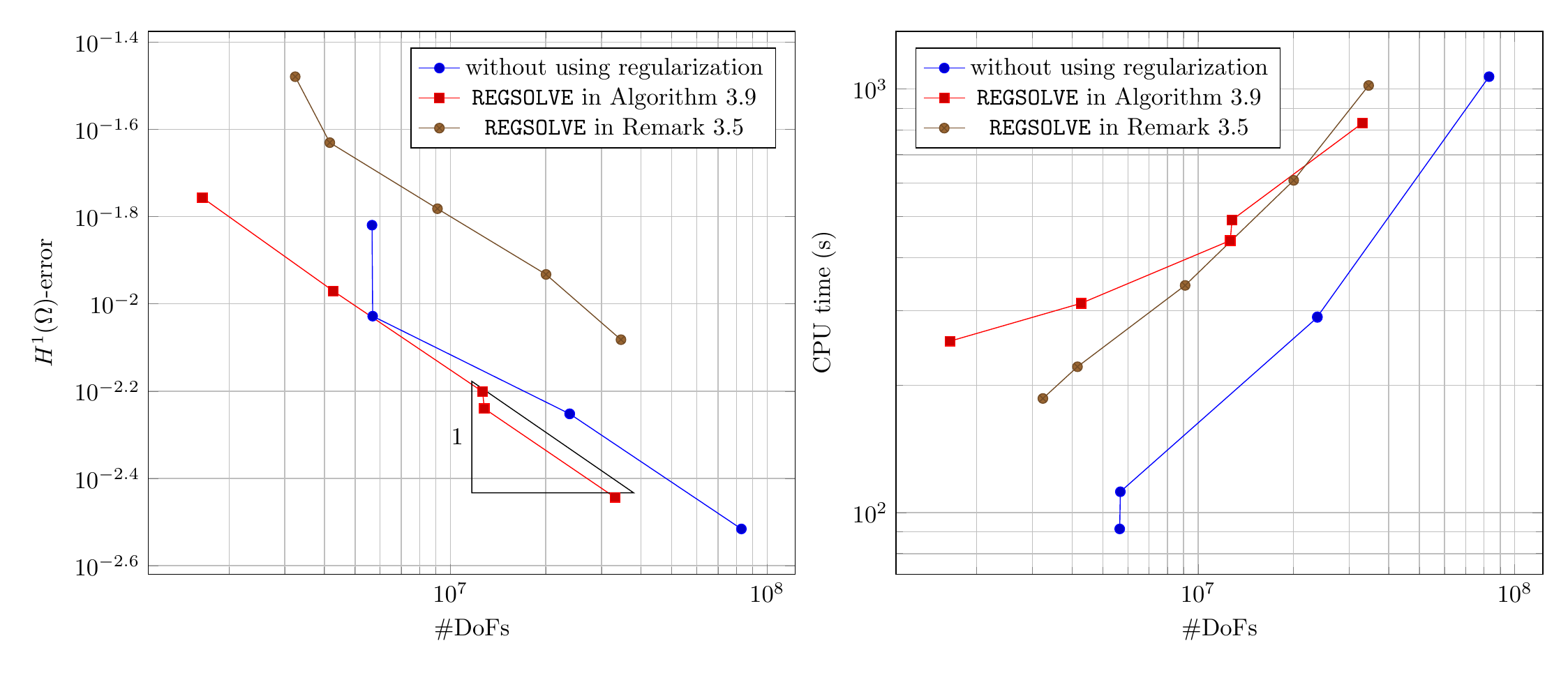}\\
    \end{center}
    \caption{Tests on a square: (left) $\|U_j-u\|_{H^1(\Omega)}$ using \emph{$\REGSOLVE$} with \emph{tensor product} $L^\infty$ for $j=8,\ldots$ and the corresponding $H^1(\Omega)$-error decay without using the regularization; (right) computational time against \#nDoFs for two adaptive algorithms. We note that the sampling points for the non-regularized algorithm at $j=10$ and $11$ are so closed that they overlap with each other.}
    \label{f:square-error-and-time}
\end{figure}

\section{Conclusion and outlook}\label{s:discussion}

We have proposed an adaptive finite element algorithm to approximate the solutions of elliptic problems with rough data approximated by regularization. Such problems are relevant in many applications ranging from fluid-structure interaction, to the modeling of biomedical applications with complex embedded domains or networks.

Our approach builds on classical results for adaptive finite element theory for $H^{-1}$ data, and for $L^2$ data. \reviewblue{In particular, we analyze the regularization of line Dirac delta distributions via convolutions with compactly supported approximated Dirac delta functions, with radius $r$}. What characterizes the regularization process is that, even if the resulting forcing term is as  regular as desired -- at fixed $r$ -- its regularity does not hold uniformly with respect to $r$.

This observation suggests that one could exploit the knowledge of the asymptotic behavior of the data regularity with respect to $r$ to construct an algorithm that \reviewblue{a priori} refines around the rough part of the forcing term, in a way that guarantees quasi optimal convergence, at least in the two dimensional case.

The resulting approximation error is split into a regularization error for $u$ and the finite element approximation error for the regularized $\uepsilon$. In this work we show how to control the dependencies between these two errors and provide an algorithm in which the error decay in the energy norm is quasi-optimal in two dimensional space and sub-optimal in three dimensional space.

Our findings are specific for the co-dimension one case, but could be easily extended to the co-dimension zero case, where the dependency of the regularity on $r$ disappears naturally, due to the intrinsic $L^2$ nature of the resulting forcing term.

\appendix

\section{Proof of Proposition~\ref{p:interface}}
Given a cell $T\in \mathcal G(\mathcal T, r)$, $T$ will be marked to refine in $\INTERFACE(\mathcal T, r)$ if $h_T>\tfrac{r}{2}$. This also holds in the process of
$\mathcal T^* = \INTERFACE(\mathcal T_0, r)$ as the parents of $T$ also satisfy the marking criteria. So $ \#(\widetilde{\mathcal T}) - \#(\mathcal T) \le \#(\mathcal T^*) - \#(\mathcal T_0)$. For any nonnegative integer $j$, denote $\widetilde {\mathcal T}_0=\mathcal T_0$ and $\widetilde{\mathcal T}_j$ the conforming refinement of $\mathcal T_{j-1}$ by bisecting cells in $\mathcal G(\gamma, \widetilde{\mathcal T}_{j-1})$ once. Let $j_r$ be the smallest integer satisfying $\mathcal G_j:=\diam(\mathcal G(\gamma, \widetilde{\mathcal T}_j))<r$. Clearly, $\widetilde{\mathcal T}_{j_r}$ is a refinement of $\mathcal T^*$ and it suffices to show that
\begin{equation}\label{i:app-target}
    \#(\widetilde{\mathcal T}_{j_r}) - \#(\mathcal T_0) \lesssim r^{1-d} .
\end{equation}

According to the assumption \eqref{a:initial-grid}, we apply the quasi-uniformity of $\mathcal G(\gamma,\mathcal T_0)$ as well as the relation $|\widetilde T| = q^{-j} |T|$ for $\widetilde T\in \mathcal G_j$ and its parent $T\in\mathcal \mathcal T_0$ to get
$
    \#(\mathcal G_j)\lesssim q^{i(1-1/d)} .
$
Summing up this estimate for $j=0,\ldots, j_r-1$ to obtain that
\[
    \#(\widetilde{\mathcal T}_{j_r}) - \#(\mathcal T_0)\le \sum_{j=0}^{j_r-1}\#(\mathcal G_j) \lesssim q^{(j_r+1)(1-1/d)}.
\]
The target estimate \eqref{i:app-target} follows from the fact $q^{-j_r/d} \sim r$ with the hidden constants depending on $q, \mathcal T_0$ and $\gamma$.

\section{Proof of Lemma~\ref{l:greedy}}
We split the proof in the following steps.

\boxed{1} Suppose that there are totally $N$ iterations executed in the while loop when ${\GREEDY}(\widetilde{\mathcal T}, F^r,\tau)$ terminates. We denote with $\{T^i\}_{i=0}^N$ the marked cells in the sequence and set $\mathcal T^i = \REFINE(\mathcal T^{i-1}, \{T^{i-1}\})$ for $i=1,\ldots,N$ and $\mathcal T^0=\widetilde{\mathcal T}$. Let
\[
    \delta := d(F^r,T^{N-1}, \mathcal T^{N-1}) =\text{argmax}\{d(F^r,T,\mathcal T^{N-1}):T\in\mathcal T^{N-1}\} .
\]
Noting that for $0\le i\le N-1$, there exists $T'\in \mathcal T^i$ such that $T^{N-1}\subset T$. Clearly, by the definition of $T^i$, we have
\begin{equation}\label{i:upper-delta}
    d(F^r,T^i,\mathcal T^i) \ge d(F^r, T',\mathcal T^i) \ge d(F^r, T^{N-1},\mathcal T^{N-1})
    =\delta .
\end{equation}
According to the definition of $T^{N-1}$, we also get
\begin{equation}\label{i:lower-delta}
    \tau \le \mathcal D(F^r,\mathcal T^{N-1}) \le \delta\sqrt{\#(\mathcal T^{N-1})} .
\end{equation}

\boxed{2} Since $F^r$ is supported in $U_r$, $T^i\subset U_{cr}$ for some constant $c\ge 1$ depending on $\constShapeRegular$. Let $\mathcal B_j\subset \{T^i\}$ be the set satisfying
\begin{equation} \label{i:class-t}
    2^{-(j+1)} |U_{cr}| < |T^i| \le 2^{-j} |U_{cr}|, \quad j\ge 0.
\end{equation}
Due to the refinement process from $\mathcal T^i$ to $\mathcal T^{i+1}$, $\{T^i\}$ are distinct from each other. This implies that by summing up the first inequality of \eqref{i:class-t} for all $T^i\in \mathcal B_j$, we obtain $\#(\mathcal B_j) < 2^{j+1}$. using the second inequality of \eqref{i:class-t} as well as \eqref{i:lower-delta}, we realize that for $T^i\in \mathcal B_j$,
\[
    \delta \le d(F^r, T^i, \mathcal T^i) = |T^i|^{1/d} \|F^r\|_{L^2(T^i)}
    \le 2^{-j/d} |U_{cr}|^{1/d} \|F^r\|_{L^2(T^i)} .
\]
We square the above estimate and sum it up for all $T^i\in \mathcal B_j$ to obtain
\[
    \delta^2\#(\mathcal B_j) \le  2^{-2j/d} |U_{cr}|^{2/d} \|F^r\|_{L^2(U_{cr})}^2
\]
The above estimate together with $\#(\mathcal B_j) < 2^{j+1}$ implies that the total marked cells can be estimated with
\begin{equation}\label{i:N-estimate}
    N =\sum_{j\ge 0}\#(\mathcal B_j) \le \sum_{j\ge 0}
    \min(2^{-2j/d}\delta^{-2} |U_{cr}|^{2/d} \|F^r\|_{L^2(U_{cr})}^2, 2^{j+1}) .
\end{equation}

\boxed{3} Since the first term of the above minimum decreases with respect to $j$, we set $j_0$ be the smallest integer such that
\begin{equation}\label{i:estimate-assumption}
    2^{j_0+1} > 2^{-2j_0/d}\delta^{-2} |U_{cr}|^{2/d} \|F^r\|_{L^2(U_{cr})}^2 .
\end{equation}
Simplifying this relation for $j_0$ we get
\[
    2^{-j_0}\lesssim \bigg(\delta^{-1}|U_{cr}|^{1/d}\|F^r\|_{L^2(U_{cr})} \bigg)^{-2d/(2+d)}.
\]
If $j_0>0$, \eqref{i:estimate-assumption}  is violated for $j=j_0-1$ and we can deduce that
\[
    2^{j_0} \lesssim \bigg( \delta^{-1}   |U_{cr}|^{1/d} \|F^r\|_{L^2(U_{cr})} \bigg)^{2d/(2+d)}.
\]
Inserting the above two estimate in \eqref{i:N-estimate}, we have
\begin{equation}\label{i:N-estimate-2}
    \begin{aligned}
        N & \le \sum_{j<j_0} 2^{j+1}
        + \delta^{-2} |U_{cr}|^{2/d} \|F^r\|_{L^2(U_{cr})}^2\sum_{j\ge j_0}  2^{-2j/d}                     \\
          & \lesssim  2^{j_0} + \bigg(\delta^{-1} |U_{cr}|^{1/d} \|F^r\|_{L^2(U_{cr})}\bigg)^2 2^{-2j_0/d} \\
          & \lesssim \bigg(\delta^{-1} |U_{cr}|^{1/d} \|F^r\|_{L^2(U_{cr})}\bigg)^{2d/(2+d)}               \\
    \end{aligned}
\end{equation}

\boxed{4} We shall further bound the above estimate by considering the $L^2(U_{cr})$-norm of $F^r$. In fact,
\[
    \begin{aligned}
        \|F^r\|_{L^2(U_{cr})}^2 & = \int_{U_{cr}}
        \bigg(\int_{\gamma}\delta^r(x-y)f(y) \diff y\bigg)^2 \diff x                                                                                         \\
                                & \le \frac{1}{r^{2d}} \|\psi\|_{L^\infty(\Real^d)}^2 \|f\|_{L^\infty(\gamma)}^2 \int_{U_{cr}} |B_r(x)\cap \gamma|^2 \diff x \\
                                & \le C\frac{1}{r^{2d}} \|\psi\|_{L^\infty(\Real^d)}^2 \|f\|_{L^\infty(\gamma)}^2 r^{2(d-1)+1}
        \lesssim r^{-1}\|f\|_{L^\infty(\gamma)}^2 ,
    \end{aligned}
\]
where for the second inequality above we used the fact that $|B_r(x)\cap\gamma|\lesssim r^{d-1}$ for $x\in D_{cr}$ and $|D_{cr}|\sim |\gamma|r$.

\boxed{5} Combining the results from Step~3 and 4 to deduce that
\begin{equation}\label{i:N-estimate-3}
    N\lesssim \delta^{-2d/(2+d)} r^{(2-d)/(2+d)} \|f\|_{L^\infty(\gamma)}^{2d/(2+d)} ,
\end{equation}
where the hidden constant above depends on $\constShapeRegular, \gamma, d$ and $\psi$. On the other hand, the complexity assumption \eqref{i:complexity-refine} implies that
\[
    \#(\mathcal T^{N-1}) - \#(\mathcal T^0) \le \constComplexity N.
\]
Hence the above relation together with \eqref{i:N-estimate-3} and \eqref{i:lower-delta} concludes that
\[
    \begin{aligned}
        \tau & \le  \delta\sqrt{\#(\mathcal T^{N-1})}                                                                                                                                    \\
             & \lesssim \frac{\sqrt{\#(\mathcal T^0) + \constComplexity N}}{N^{(2+d)/(2d)}} r^{1/d-1/2} \|f\|_{L^\infty(\gamma)} \lesssim N^{-1/d}r^{1/d-1/2} \|f\|_{L^\infty(\gamma)} ,
    \end{aligned}
\]
which implies the target estimate.

\boxed{6} If $j_0=0$, we directly get $\delta^{-1}|U_{cr}|^{1/d} \|F^r\|_{L^2(U_{cr})} \lesssim 1$. Starting form \eqref{i:N-estimate}, we apply the result in Step~4 to write
\[
    \begin{aligned}
        N & \lesssim  \delta^{-2}|U_{cr}|^{2/d} \|F^r\|_{L^2(U_{cr})}^2\sum_{j\ge 0} 2^{-2j/d}                              \\
          & \lesssim \delta^{-2}|U_{cr}|^{2/d} \|F^r\|_{L^2(U_{cr})}^2                                                      \\
          & \le \delta^{-1} |U_{cr}|^{1/d} \|F^r\|_{L^2(U_{cr})} \lesssim \delta^{-1} r^{1/d-1/2}\|f\|_{L^\infty(\gamma)} .
    \end{aligned}
\]
Therefore, we again obtain the target estimate following the argument in Step~5. The proof is complete.

\bibliographystyle{siamplain}
\bibliography{references}
\end{document}

%% file: tmp_paper_header.tex
\title{Adaptive finite element approximations for elliptic problems using regularized forcing data\thanks{Draft date: \today
\funding{This work was partially supported by the National Research Projects (PRIN  2017) ``Numerical Analysis for Full and Reduced Order Methods for the efficient and accurate solution of complex systems governed by Partial Differential Equations'', funded by the Italian Ministry of Education, University, and Research.}}}

\author{Luca Heltai\thanks{Mathematics Area, SISSA -- International School for Advanced Studies, via Bonomea 265, 34136, Trieste, Italy (\email{luca.heltai@sissa.it}, \email{wenyu.lei@sissa.it}).}
\and Wenyu Lei\footnotemark[2]
}

\headers{AFEM for regularized elliptic problems}{L.~Heltai and W.~Lei}

%% file: tmp_paper_abstract.tex
    \begin{abstract}
        We propose an adaptive finite element algorithm to approximate solutions
        of elliptic problems whose forcing data is locally defined and is
        approximated by regularization (or mollification). We show that
        the energy error decay is quasi-optimal in two dimensional space and
        sub-optimal in three dimensional space. Numerical simulations are
        provided to confirm our findings.
    \end{abstract}

    \begin{keywords}
        Finite elements, interface problems, immersed boundary method, Dirac delta approximations, a posteriori error estimates, adaptivity
    \end{keywords}

    \begin{AMS}
        65N15,                    %
        65N30.                    %
        65N50,                    %
    \end{AMS}